\documentclass[11pt,dvipsnames]{amsart}

\title[Realizing Compatible Pairs of Transfer Systems]{Realizing Compatible Pairs of Transfer Systems by Combinatorial $N_\infty$-Operads}

\author[Chan]{David Chan}
\address[Chan]{Michigan State University}
\email{chandav2@msu.edu}
\author[Cho]{Myungsin Cho}
\address[Cho]{Columbia University}
\email{mc5942@columbia.edu}
\author[Mehrle]{David Mehrle}
\address[Mehrle]{University of Kentucky}
\email{davidm@uky.edu}
\author[Ocal]{Pablo S. Ocal}
\address[Ocal]{Okinawa Institute of Science and Technology}
\email{pablo.ocal@oist.jp}
\author[Osorno]{Ang\'elica M. Osorno}
\address[Osorno]{Reed College}
\email{aosorno@reed.edu}
\author[Szczesny]{Ben Szczesny}
\address[Szczesny]{Ohio University}
\email{ben.szczesny@ohio.edu}
\author[Verdugo]{Paula Verdugo}
\address[Verdugo]{Max Planck Institute for Mathematics}
\email{verdugo@mpim-bonn.mpg.de}

\usepackage{amsmath,amssymb,amsfonts} 
\usepackage{bm}
\usepackage{booktabs}
\usepackage[shortlabels]{enumitem}
    \setlist[enumerate,1]{label = (\alph*)}
\usepackage[margin=1in]{geometry}
\usepackage[pagebackref, colorlinks, citecolor=PineGreen, linkcolor=PineGreen]{hyperref} 

\usepackage{mathtools}
\usepackage[final]{microtype}
\usepackage{tikz-cd} 
\usepackage[textwidth=22mm, textsize=tiny]{todonotes} 

\usepackage[nameinlink,capitalise,noabbrev]{cleveref}

\numberwithin{equation}{section} 

\theoremstyle{plain}
\newtheorem{theorem}[equation]{Theorem}
\newtheorem{corollary}[equation]{Corollary}
\newtheorem{proposition}[equation]{Proposition}
\newtheorem{lemma}[equation]{Lemma}
\newtheorem{introtheorem}{Theorem}

\newtheorem{conjecture}[equation]{Conjecture}

\theoremstyle{definition}
\newtheorem{definition}[equation]{Definition}
\newtheorem{example}[equation]{Example}
\newtheorem{remark}[equation]{Remark}
\newtheorem{notation}[equation]{Notation}

\newtheorem{property}{Property}
\crefname{lemma}{Lemma}{Lemmas}
\crefname{theorem}{Theorem}{Theorems}
\crefname{definition}{Definition}{Definitions}
\crefname{proposition}{Proposition}{Propositions}
\crefname{remark}{Remark}{Remarks}
\crefname{corollary}{Corollary}{Corollaries}
\crefname{equation}{}{} 
\crefname{construction}{Construction}{Constructions}
\crefname{example}{Example}{Examples}
\crefname{subsection}{Subsection}{Subsections}
\crefname{notation}{Notation}{Notations}
\crefname{introtheorem}{Theorem}{Theorems}
\crefname{conjecture}{Conjecture}{Conjectures}
\crefname{property}{Property}{Properties}

\newtheorem*{definition*}{Definition}

\renewcommand{\phi}{\varphi}

\newcommand{\RR}{\mathbb{R}}

\newcommand{\cO}{\mathcal{O}}
\newcommand{\cP}{\mathcal{P}}
\newcommand{\cQ}{\mathcal{Q}}
\newcommand{\cK}{\mathcal{K}}
\newcommand{\cL}{\mathcal{L}}
\newcommand{\cI}{\mathcal{I}}

\newcommand{\cT}{\mathcal{T}}

\newcommand{\Top}{\mathrm{Top}}
\newcommand{\Cat}{\mathrm{Cat}}
\newcommand{\cat}[1]{\mathcal{#1}}
\newcommand{\Fin}{\mathrm{Fin}}

\newcommand{\id}{\operatorname{id}}
\newcommand{\im}{\operatorname{im}}
\newcommand{\Map}{\operatorname{Map}}

\DeclareMathOperator*{\colim}{colim}
\DeclareMathOperator{\Hull}{Hull}
\newcommand{\Res}{\operatorname{Res}}
\newcommand{\Set}{\mathrm{Set}}

\newcommand{\Ind}{\mathrm{Ind}}

\newcommand{\Coind}{\mathrm{Coind}}

\DeclarePairedDelimiterX\set[1]\lbrace\rbrace{#1} 

\newcommand{\sthreescale}{0.5}

\newcommand{\sthreetriv}{
	\begin{tikzpicture}[scale=\sthreescale,baseline=(c2b).base]
		\clip (-1.75,-0.5) rectangle (1.75,2.5);
	
        \node[fill=white,scale=\sthreescale] (e) at (0,0) {$1$};
        \node[fill=white,scale=\sthreescale] (c2a) at (-1.5,1) {$H_1$};
        \node[fill=white,scale=\sthreescale] (c2b) at (0,1) {$H_2$};
        \node[fill=white,scale=\sthreescale] (c2c) at (1.5,1) {$H_3$};    
        \node[fill=white,scale=\sthreescale] (c3) at (0.7,1.2)
            {$C_3$};
        \node[fill=white,scale=\sthreescale] (s3) at (0,2)
            {$\Sigma_3$};
	\end{tikzpicture}
}

\newcommand{\sthreea}{
	\begin{tikzpicture}[scale=\sthreescale,baseline=(c2b).base]
		\clip (-1.75,-0.5) rectangle (1.75,2.5);
	
        \node[fill=white,scale=\sthreescale] (e) at (0,0) {$1$};
        \node[fill=white,scale=\sthreescale] (c2a) at (-1.5,1) {$H_1$};
        \node[fill=white,scale=\sthreescale] (c2b) at (0,1) {$H_2$};
        \node[fill=white,scale=\sthreescale] (c2c) at (1.5,1) {$H_3$};    
        \node[fill=white,scale=\sthreescale] (c3) at (0.7,1.2)
            {$C_3$};
        \node[fill=white,scale=\sthreescale] (s3) at (0,2)
            {$\Sigma_3$};

        \draw (e) -- (c2a);
        \draw (e) -- (c2b);
        \draw (e) -- (c2c);
	\end{tikzpicture}
}

\newcommand{\sthreeb}{
	\begin{tikzpicture}[scale=\sthreescale,baseline=(c2b).base]
		\clip (-1.75,-0.5) rectangle (1.75,2.5);
	
        \node[fill=white,scale=\sthreescale] (e) at (0,0) {$1$};
         \node[fill=white,scale=\sthreescale] (c2a) at (-1.5,1) {$H_1$};
        \node[fill=white,scale=\sthreescale] (c2b) at (0,1) {$H_2$};
        \node[fill=white,scale=\sthreescale] (c2c) at (1.5,1) {$H_3$};   
        \node[fill=white,scale=\sthreescale] (c3) at (0.7,1.2)
            {$C_3$};
        \node[fill=white,scale=\sthreescale] (s3) at (0,2)
            {$\Sigma_3$};

        \draw (e) -- (c3);
	\end{tikzpicture}
}

\newcommand{\sthreeab}{
	\begin{tikzpicture}[scale=\sthreescale,baseline=(c2b).base]
		\clip (-1.75,-0.5) rectangle (1.75,2.5);
	
        \node[fill=white,scale=\sthreescale] (e) at (0,0) {$1$};
         \node[fill=white,scale=\sthreescale] (c2a) at (-1.5,1) {$H_1$};
        \node[fill=white,scale=\sthreescale] (c2b) at (0,1) {$H_2$};
        \node[fill=white,scale=\sthreescale] (c2c) at (1.5,1) {$H_3$};  
        \node[fill=white,scale=\sthreescale] (c3) at (0.7,1.2)
            {$C_3$};
        \node[fill=white,scale=\sthreescale] (s3) at (0,2)
            {$\Sigma_3$};

        \draw (e) -- (c2a);
        \draw (e) -- (c2b);
        \draw (e) -- (c2c);
        \draw (e) -- (c3);
	\end{tikzpicture}
}

\newcommand{\sthreec}{
	\begin{tikzpicture}[scale=\sthreescale,baseline=(c2b).base]
		\clip (-1.75,-0.5) rectangle (1.75,2.5);
	
        \node[fill=white,scale=\sthreescale] (e) at (0,0) {$1$};
        \node[fill=white,scale=\sthreescale] (c2a) at (-1.5,1) {$H_1$};
        \node[fill=white,scale=\sthreescale] (c2b) at (0,1) {$H_2$};
        \node[fill=white,scale=\sthreescale] (c2c) at (1.5,1) {$H_3$};   
        \node[fill=white,scale=\sthreescale] (c3) at (0.7,1.2)
            {$C_3$};
        \node[fill=white,scale=\sthreescale] (s3) at (0,2)
            {$\Sigma_3$};

        \draw (e) -- (c2a);
        \draw (e) -- (c2b);
        \draw (e) -- (c2c);
        \draw (c3) -- (s3);
	\end{tikzpicture}
}

\newcommand{\sthreeabc}{
	\begin{tikzpicture}[scale=\sthreescale,baseline=(c2b).base]
		\clip (-1.75,-0.5) rectangle (1.75,2.5);
	
        \node[fill=white,scale=\sthreescale] (e) at (0,0) {$1$};
         \node[fill=white,scale=\sthreescale] (c2a) at (-1.5,1) {$H_1$};
        \node[fill=white,scale=\sthreescale] (c2b) at (0,1) {$H_2$};
        \node[fill=white,scale=\sthreescale] (c2c) at (1.5,1) {$H_3$};  
        \node[fill=white,scale=\sthreescale] (c3) at (0.7,1.2)
            {$C_3$};
        \node[fill=white,scale=\sthreescale] (s3) at (0,2)
            {$\Sigma_3$};

        \draw (e) -- (c2a);
        \draw (e) -- (c2b);
        \draw (e) -- (c2c);
        \draw (e) -- (c3);
        \draw (e) edge[bend left=30] (s3);
        \draw (c3) -- (s3);
	\end{tikzpicture}
}

\newcommand{\sthreed}{
	\begin{tikzpicture}[scale=\sthreescale,baseline=(c2b).base]
		\clip (-1.75,-0.5) rectangle (1.75,2.5);
	
        \node[fill=white,scale=\sthreescale] (e) at (0,0) {$1$};
        \node[fill=white,scale=\sthreescale] (c2a) at (-1.5,1) {$H_1$};
        \node[fill=white,scale=\sthreescale] (c2b) at (0,1) {$H_2$};
        \node[fill=white,scale=\sthreescale] (c2c) at (1.5,1) {$H_3$};   
        \node[fill=white,scale=\sthreescale] (c3) at (0.7,1.2)
            {$C_3$};
        \node[fill=white,scale=\sthreescale] (s3) at (0,2)
            {$\Sigma_3$};

        \draw (e) -- (c2a);
        \draw (e) -- (c2b);
        \draw (e) -- (c2c);
        \draw (e) -- (c3);
        \draw (e) edge[bend left=30] (s3);
	\end{tikzpicture}
}

\newcommand{\sthreee}{
	\begin{tikzpicture}[scale=\sthreescale,baseline=(c2b).base]
		\clip (-1.75,-0.5) rectangle (1.75,2.5);
	
        \node[fill=white,scale=\sthreescale] (e) at (0,0) {$1$};
        \node[fill=white,scale=\sthreescale] (c2a) at (-1.5,1) {$H_1$};
        \node[fill=white,scale=\sthreescale] (c2b) at (0,1) {$H_2$};
        \node[fill=white,scale=\sthreescale] (c2c) at (1.5,1) {$H_3$};    
        \node[fill=white,scale=\sthreescale] (c3) at (0.7,1.2)
            {$C_3$};
        \node[fill=white,scale=\sthreescale] (s3) at (0,2)
            {$\Sigma_3$};

        \draw (e) -- (c2a);
        \draw (e) -- (c2b);
        \draw (e) -- (c2c);
        \draw (e) -- (c3);
        \draw (e) edge[bend left=30] (s3);
        \draw (c2a) -- (s3);
        \draw (c2b) -- (s3);
        \draw (c2c) -- (s3);
	\end{tikzpicture}
}

\newcommand{\sthreeall}{
	\begin{tikzpicture}[scale=\sthreescale,baseline=(c2b).base]
		\clip (-1.75,-0.5) rectangle (1.75,2.5);
	
        \node[fill=white,scale=\sthreescale] (e) at (0,0) {$1$};
         \node[fill=white,scale=\sthreescale] (c2a) at (-1.5,1) {$H_1$};
        \node[fill=white,scale=\sthreescale] (c2b) at (0,1) {$H_2$};
        \node[fill=white,scale=\sthreescale] (c2c) at (1.5,1) {$H_3$};   
        \node[fill=white,scale=\sthreescale] (c3) at (0.7,1.2)
            {$C_3$};
        \node[fill=white,scale=\sthreescale] (s3) at (0,2)
            {$\Sigma_3$};

        \draw (e) -- (c2a);
        \draw (e) -- (c2b);
        \draw (e) -- (c2c);
        \draw (e) -- (c3);
        \draw (e) edge[bend left=30] (s3);
        \draw (c2a) -- (s3);
        \draw (c2b) -- (s3);
        \draw (c2c) -- (s3);
        \draw (c3) -- (s3);
	\end{tikzpicture}
}

\date{}



\begin{document}

\begin{abstract} 
    We investigate how the notions of pairings of operads of May and compatible pairs of indexing systems of Blumberg--Hill relate via the correspondence between indexing systems and $N_{\infty}$-operads. We show that a pairing of operads induces a pairing on the associated indexing systems. Conversely, we show that in many cases, compatible pairs of indexing systems can be realized by a pairing of $N_{\infty}$-operads. 
\end{abstract}

\maketitle

\setcounter{tocdepth}{1}
\tableofcontents

\section{Introduction}\label{sec:introduction}

Algebraic topology is commonly understood to be the study of topological objects by means of their algebraic data, suggesting a dichotomy with topology on one side and algebra on the other. In practice, however, the distinction is not so stark. Endowing the topological objects themselves with algebraic structure is nowadays common practice in algebraic topology. 
A convenient framework for working with such structures is that of \textit{operads}. 
Operads parameterize operations on a topological space and 
encode the homotopies describing properties like associativity or commutativity. 
First formalized by May \cite{May:Geometry}, operads have found many applications 
in homotopy theory and beyond \cite{Keller,Fresse1,Fresse2,MarklShniderStasheff}. 

The most structured kind of operad, known as an \textit{$E_\infty$-operad}, encodes operations which are both associative and commutative up to all higher homotopies. While different $E_\infty$-operads are suitable for different contexts, they all are weakly homotopy equivalent, as their constituent spaces are contractible.
However, in the presence of a group action, the notion of an $E_\infty$-operad breaks down 
insofar as not all equivariant operads that are non-equivariantly $E_\infty$ are equivalent. 
This observation led Blumberg and Hill to the definition of $N_\infty$ operads and their systematic study \cite{BlumbergHillOperadic}. Surprisingly, homotopy types of 
$N_\infty$-operads are classified by purely combinatorial objects called 
\textit{transfer systems} (or an equivalent incarnation called \textit{indexing systems}) 
that depend only on the group of equivariance~\cite{BlumbergHillOperadic,BalchinBarnesRoitzheim,RubinDetecting}. 
This connection between equivariant homotopy theory and combinatorics reduces a delicate topological question about equivariant operads to a tractable, combinatorial problem. 

\medskip

Often, a space has two separate yet compatible families of operations described by two operads $\cP$ and $\cQ$. This is roughly analogous to the structure of a ring: an addition and a multiplication compatible via a distributive law.
A \textit{pairing of operads} or \textit{operad pairing} (\cite{MayEringBook}, \cref{def:operad_pairing}) 
describes an operadic distributive law for $\cP$ over $\cQ$ via a sequence of structure maps relating the different levels of the operad. 

In this paper, we study the relationship between transfer systems and pairings of $N_\infty$-operads. 
Recent work of Blumberg--Hill~\cite{BlumbergHillBiincomplete} and the first author~\cite{ChanBiincomplete} 
considers the kinds of distributive laws which are sensible to consider on the level of transfer systems 
(or on the equivalent indexing systems, see \cref{def: indexing system pair}). Further work has considered the combinatorics of compatible pairs of transfer systems in the specific cases of groups of the form $C_{p^n}$ \cite{HillMengLi} and $C_{p^nq^m}$ \cite{MORSVZ}. However, the connection between this purely combinatorial definition and operads has only been considered in limited cases. Our first main theorem establishes that a pairing of $N_{\infty}$-operads yields a distributive law on the associated transfer systems. 

\begin{introtheorem}[\cref{prop:operad-action-induces-compatible-pair}]
	\label{introthm: main 1}
	Let $\cP$ and $\cQ$ be $N_\infty$-operads with transfer systems $\cT_\cP$ and $\cT_\cQ$, respectively. If there is a pairing of operads between $\cP$ and $\cQ$, then $(\cT_\cP,\cT_\cQ)$ is a compatible pair of transfer systems.
\end{introtheorem}

This theorem provides an explicit homotopical obstruction to the existence of pairings between $N_{\infty}$-operads.  In particular if $\cP$ and $\cQ$ are two $N_{\infty}$-operads with incompatible transfer systems then the theorem implies that there does not exist a pairing of $\cP$ and $\cQ$.  Considering the case where $\cP=\cQ$, it follows that an operad can admit a self-pairing only if the associated transfer system $\cT_{\cP}$ is compatible with itself.  Such a transfer system is called \emph{saturated} (\cref{remark saturated transfer systems}), and this class has seen significant interest \cite{RubinDetecting,HMOO,macbrough,Bannwart,BS2025} because of their relation to \emph{linear isometry operads}.  See \cref{def: linear isometry operad} for the definition of the linear isometry operads. As most transfer systems for a given group are not saturated, it follows that most $N_{\infty}$-operads cannot admit self-pairings.

To discuss a converse to \cref{introthm: main 1}, we introduce the following terminology: 

\begin{definition}
    A compatible pair of transfer systems $(\cT_1, \cT_2)$ is \textit{realizable} if there exists a pairing of $N_\infty$-operads between $\cO_1$ and $\cO_2$ whose associated transfer systems are $\cT_1$ and $\cT_2$, respectively. 
\end{definition}

The seemingly simpler question of whether or not every transfer system is realized by an $N_{\infty}$-operad (absent the pairing conditions) is non-trivial. This was a conjecture of Blumberg--Hill \cite[paragraph after Theorem 1.2]{BlumbergHillOperadic} which eventually saw several different solutions \cite{BonventrePereira,GutierrezWhite,RubinCombinatorial}. 

Since every transfer system is realized by an $N_\infty$-operad, one might be tempted to think that realizing a compatible pair of transfer systems is simply a matter of checking whether or not the associated operads comprise a pairing. Unfortunately, they rarely do. 
A major difficulty is the fact that pairing of operads is not a homotopy invariant notion. 
For example, although all (non-equivariant) $E_\infty$-operads are homotopy equivalent, the construction of a pairing relies significantly on the specific choices of operads; there is a canonical pairing of the linear isometries and Steiner operads \cite{Steiner}, but it is an open question whether or not the Steiner operad admits a pairing with itself. Despite this obstruction, we make the following conjecture.

\begin{conjecture}\label{conjecture: realizability}
    For every finite group $G$, every compatible pair of $G$-transfer systems is realizable.
\end{conjecture}

As evidence towards this conjecture, we produce several families of examples. 

\begin{introtheorem}[\cref{theorem:mainthm}]
Let $G$ be a finite group. Let $(\cT_m,\cT_a)$ be a compatible pair of $G$-transfer systems, with $\cT_a$ complete. Then there exists a pairing of $N_\infty$-operads between $\cO_m$ and $\cO_a$ whose associated transfer systems are $\cT_m$ and $\cT_a$ respectively.
\end{introtheorem}

\begin{introtheorem}[\cref{lemma coinduced operad pairs}]
    Let $H$ be a subgroup of a finite group $G$ and let $(\cT_m,\cT_a)$ be a compatible pair of $H$-transfer systems. Let $\cT_m^G$ and $\cT_a^G$ be the coinduced $G$-transfer systems, in the sense of \cref{def coinduced transfer system}. Then $(\cT_m^G, \cT_a^G)$ is a compatible pair of transfer systems, and if $(\cT_m, \cT_a)$ is realizable, so is $(\cT_m^G,\cT_a^G)$.
\end{introtheorem}

We also verify that the canonical pairing of operads between the linear isometries and Steiner operads \cite{Steiner} extends to their $G$-equivariant versions. 

\begin{introtheorem}[\cref{example of realizable pairs 1}]
    Let $G$ be a finite group and let $U$ be a $G$-universe. Let $\cT_{\cL_G(U)}$ be the transfer system associated to the equivariant linear isometries operad $\cL_G(U)$ and let $\cT_{\cK_G(U)}$ be the transfer system associated to the equivariant Steiner operad $\cK_G(U)$. Then the pair $(\cT_{\cL_G(U)},\cT_{\cK_G(U)})$ is realized by the pairing of operads between $\cL_G(U)$ and $\cK_G(U)$.
\end{introtheorem}

We also provide tools for building new realizable pairs from known cases in \cref{subsec:realizable_pairs_from_known}. Finally, we show that \cref{conjecture: realizability} can be reduced to the realizability of certain types of compatible pairs of transfer systems. 

\begin{introtheorem}[\cref{conjecture reduced}]
    \cref{conjecture: realizability} is equivalent to the claim that $(\Hull(\cT), \cT)$ is realizable for any transfer system $\cT$, where $\Hull(\cT)$ is the \emph{multiplicative hull} of $\cT$ \cite[Proposition 7.84]{BlumbergHillBiincomplete}. 
\end{introtheorem}

Constructing examples of pairings of operads is a difficult problem. In the course of proving the theorems above we produce a new method of constructing operad pairings which should be of independent interest, as it can be used to produce new examples of operad pairings, even non-equivariantly.  The key input to this construction is the observation that one can produce operads, in the category of sets, from monoids.  More precisely, given a monoid $M$, there is an operad $\mathcal{O}(M)$ with $\mathcal{O}(M)(n)=M^n$, and this construction is functorial in maps of monoids.

The resulting operads are symmetric but do not have a free action of the symmetric groups (they are not \textit{$\Sigma$-free}). However, if the monoid $M$ is equipped with an \emph{intersection structure} (\cref{definition: intersection monoid}), then one can modify the standard construction to produce a $\Sigma$-free operad $\mathcal{O}^{\vee}(M)$.  The notion of intersection monoids was introduced in independent work of the sixth author for this purpose \cite{szczesny2025realizingtransfersystemssuboperads}.  

Building on this, we define the notion of a \emph{pairing of intersection monoids}, which mimics the data of a pairing of operads.  We show that every pairing of intersection monoids produces a pairing on the associated symmetric operads.
\begin{introtheorem}[\cref{thm: operad pairing from compatible monoids}]
    Given two intersection monoids $M$ and $N$ and a pairing $\xi\colon M\times N\to N$ there is a pairing of operads between $\mathcal{O}^{\vee}(M)$ and $\mathcal{O}^{\vee}(N)$.
\end{introtheorem}

We use this theorem to produce pairings of operads in the category of sets, where we have fairly explicit control on the structure of the operads.  Following ideas of Rubin, we upgrade these to pairings of $N_{\infty}$-operads.  Moreover, we gain control on the transfer systems associated to the resulting $N_{\infty}$-operads using recent results of the sixth author. 

\subsection{Notation}
\label{notation}
Throughout this paper, we use the letter $G$ for a finite group. For a natural number $n$, we write $\underline{n} \coloneqq \{1,2,\ldots,n\}$. We also follow the conventions of the literature and write $\Sigma_n$ for the symmetric group on $n$ letters. 

We will also frequently refer to the categories $\Set$ of sets and functions, $\Top$ of (compactly generated and weakly Hausdorff) topological spaces and continuous maps, and $\Cat$ of small categories and functors. We write $\Set^G$, $\Top^G$, and $\Cat^G$ for the respective categories of $G$-objects and $G$-equivariant maps in each.

\subsection*{Acknowledgements}
The authors warmly thank Mike Hill for helpful conversations and Jonathan Rubin for discussions that led to the results in \cref{subsec:realizable_pairs_from_known}.  
We would also like to thank the American Mathematical Society (AMS) for hosting 
the 2024 Mathematical Research Community (MRC) on Homotopical Combinatorics, 
where this research project began. This AMS MRC was supported by NSF grant DMS-1916439. 
DC was partially supported by NSF grant DMS-2135960. 
MC was partially supported by NSF grant DMS-2405030, DMS-2052846, and DMS-2104348.
DM was partially supported by NSF grant DMS-2135884. 
PSO was partially supported by JSPS KAKENHI grant JP25K17242. 
AO was partially supported by NSF grant DMS–2204365. 

\section{Operads and \texorpdfstring{$N_\infty$}{N-infinity}-operads}

We recall the definitions and present some examples of operads (\cref{subsec:background_operads}) and $N_\infty$-operads (\cref{subsec:background_Ninfinity}). For more comprehensive references, see \cite{May:Geometry,BlumbergHillOperadic}. 

\subsection{Operads}\label{subsec:background_operads}
The main purpose of an operad is to encode different additive and multiplicative structures, together with properties of these operations, like associativity and commutativity, up to coherent homotopies. In particular, an algebra over an operad is a space with many possible choices of operations, which are precisely parametrized by the points in the different levels of the operad.  

Let $(\cat{C}, \times, \ast)$ denote one of the Cartesian monoidal categories $(\Set, \times, \ast)$, $(\Top,\times,\ast)$, $(\Cat,\times,\ast)$, or their $G$-equivariant versions for a finite group $G$. In particular, we assume that the objects of $\cat{C}$ have underlying sets and morphisms can be expressed as functions with extra structure.

Before we discuss operads, we first introduce some notation.

\begin{notation} 
	\label{notation:block permutation}
	Let $k_1, \ldots, k_n$ be positive integers and let $k_+ = \sum_{i=1}^n k_i$. We identify the finite set $\underline{k_+}$ with the coproduct 
    $\underline{k_1}  \amalg \dots \amalg \underline{k_n}$ 
    via the bijection that orders all terms of $\underline{k_i}$ before any term of $\underline{k_j}$ for $i < j$, while maintaining the natural ordering within each $\underline{k_i}$.     
	\begin{enumerate}[(a)]
		\item Let $\sigma \in \Sigma_n$.  We write 
		\[
			\sigma_+(k_1, \dots, k_n) \in \Sigma_{k_+}
		\]
		for the permutation obtained by writing
		\(
            \underline{k_+} \cong \underline{k_1} \amalg  \dots \amalg \underline{k_n} 
		\) 
        via the bijection described above
		and permuting the summands according to $\sigma$.
		\item Let $\tau_i \in \Sigma_{k_i}$ for each $i$, $1 \leq i \leq n$. We write 
		\[
			\tau_1  \oplus \dots \oplus \tau_n \in \Sigma_{k_+}
		\]
		for the permutation of $\underline{k_+}$ obtained by writing
		\(
            \underline{k_+} \cong \underline{k_1} \amalg  \dots \amalg \underline{k_n}
		\) 
        via the bijection described above
		and acting on $\underline{k_i}$ with $\tau_i$ for each $i \in \underline{n}$.	
	\end{enumerate}
\end{notation}

\begin{definition}[{\cite[Definition 2.1]{RubinCombinatorial} and~\cite[Definition 1.1]{May:Geometry}}]
\label{def:reduced-operad}
    A \emph{(reduced) symmetric operad $\cO$ in 
    $\cat{C}$} is a collection of objects $\cO(m)$ in $\cat{C}$ 
    for all non-negative integers $m$, 
    together with structure maps 
    \[
        \gamma = \gamma_{n,k_1,\dots,k_n}\colon \cO(n)\times \cO(k_1)\times \cdots\times\cO(k_n)\to \cO(k_+) 
    \]
    for all integers $n, k_1, \dots, k_n \geq 0$ 
    subject to the following conditions. 
    \begin{enumerate}
        \item (\emph{reduced}) The zero-th object is the monoidal unit $\cO(0)=\ast$. 
	    \item (\emph{identity}) There is an element $\id \in \cO(1)$ such that for any $x\in \cO(k)$ we have $\gamma(\id;x) = x$.  When $\mathcal{O}$ is an operad in $\Set^G$,  $\Top^G$, or $\Cat^G$, we demand that $\id$ is $G$-fixed. 
        \item\label{operad associativity} 
        (\emph{associativity}) For any non-negative integer $j$, non-negative integers $n_\ell$ for $1 \leq \ell \leq j$, and non-negative integers $k_{\ell, i}$ for $1 \leq \ell \leq j$, $1 \leq i \leq n_\ell$, and elements $x_{\ell}\in \cO(n_{\ell})$ and $y_{\ell,i}\in \cO(k_{\ell,i})$, and $z\in \cO(j)$, we have
        \[
            \gamma\bigg(\gamma\big(z;(x_\ell)_{\ell =1}^j\big);\big((y_{\ell,i})_{i=1}^{n_\ell}\big)_{\ell=1}^j\bigg) = \gamma\bigg(z;\gamma\big(x_1;(y_{1,i})_{i=1}^{n_1}\big),\dots,\gamma\big(x_j;(y_{j,i})_{i=1}^{n_j}\big)\bigg).
        \]
        \item\label{operad symmetry} 
        (\emph{symmetry}) There is a right action of the symmetric group 
	    $\Sigma_n$ on $\cO(n)$. This action is subject to the conditions
        \medskip
        \begin{enumerate}[label=(\roman*), itemsep=1em]
            \item \label{operad symmetry 1}
            \(
            \gamma(x \cdot \sigma ; y_1,\dots,y_n) 
                =  
            \gamma(x ; y_{\sigma^{-1}(1)},\dots,y_{\sigma^{-1}(n)}) 
                \cdot \sigma_+(k_1,\dots,k_n) 
            \)
            \item \label{operad symmetry 2}
            \(
            \gamma(x ; y_1 \cdot \tau_1,\dots,y_n \cdot \tau_n) 
                 = 
            \gamma(x ; y_{1},\dots,y_{n}) \cdot (\tau_1 \oplus \dots \oplus \tau_n) 
            \)
        \end{enumerate}
        \medskip
        for any $x\in \cO(n)$, $y_{i}\in \cO(k_i)$, $\sigma\in \Sigma_n$ and $\tau_i\in \Sigma_{k_i}$.
    \end{enumerate}
\end{definition}

\begin{remark}
The formulas (c) and (d) above can be expressed as the commutativity of the following diagrams.
\begin{equation}
\tag*{\ref{operad associativity}}
\begin{tikzcd}
	\displaystyle{
        \cO(j)\times\prod_{\ell=1}^j\left(\cO(n_\ell) 
            \times 
        \prod_{i=1}^{n_\ell}\cO(k_{\ell,i})\right)
    } 
        &
        & 
    \displaystyle{\cO(j)\times\prod_{\ell=1}^j \cO(k_{\ell,+})} 
        \\
	\displaystyle{
        \left(\cO(j)\times\prod_{\ell=1}^j \cO(n_\ell)\right) 
        \times 
        \prod_{\ell=1}^j\prod_{i=1}^{n_\ell}\cO(k_{\ell,i})
    } 
        \\
    \displaystyle{
        \cO(n_+)\times\prod_{\ell=1}^j\prod_{i=1}^{n_\ell}\cO(k_{\ell,i})
    } 
        &
        & 
    {\cO(k_{+})}
	   \arrow["{\id\times \left( \prod \gamma \right)}", from=1-1, to=1-3]
    	\arrow["\cong"', from=1-1, to=2-1]
    	\arrow["\gamma", from=1-3, to=3-3]
    	\arrow["{\gamma\times \id}"', from=2-1, to=3-1]
    	\arrow["\gamma"', from=3-1, to=3-3]
\end{tikzcd}
\end{equation}

\begin{equation}
\tag*{\ref{operad symmetry}\,\ref{operad symmetry 1}}
\begin{tikzcd}[column sep=small]
	\displaystyle{\cO(n)\times \prod_{i=1}^n \cO(k_i)} 
        &
        & 
    \displaystyle{\cO(n)\times \prod_{i=1}^n \cO(k_{\sigma^{-1}(i)})} 
        \\
        &
        & 
    {\cO(k_+)} 
        \\
	\displaystyle{\cO(n)\times \prod_{i=1}^n \cO(k_i)} 
        &
        & 
    {\cO(k_+)}
	   \arrow["\cong", from=1-1, to=1-3]
	   \arrow["{\sigma\times \id}"', from=1-1, to=3-1]
	   \arrow["\gamma", from=1-3, to=2-3]
	   \arrow["{\sigma_+(j_1,\dots,j_n)}", from=2-3, to=3-3]
	   \arrow["\gamma"', from=3-1, to=3-3]
\end{tikzcd}
\end{equation}
\begin{equation}
\tag*{\ref{operad symmetry}\,\ref{operad symmetry 2}}
\begin{tikzcd}[column sep=small]
	\displaystyle{\cO(n)\times \prod_{i=1}^n \cO(k_i)} 
        &
        & 
    {\cO(k_+)} 
        \\
	\displaystyle{\cO(n)\times \prod_{i=1}^n \cO(k_i)} 
        &
        & 
    {\cO(k_+)}
	   \arrow["\gamma", from=1-1, to=1-3]
    	\arrow["{\id\times \left( \prod\tau_i \right)}"', from=1-1, to=2-1]
    	\arrow["{(\tau_1\oplus\dots\oplus\tau_n)}", from=1-3, to=2-3]
    	\arrow["\gamma"', from=2-1, to=2-3]
\end{tikzcd}
\end{equation}
\end{remark}

\begin{definition}
    A \emph{morphism of operads} $f\colon \cO\to \cO'$ is a collection of $\Sigma_n$-equivariant morphisms $f_n\colon \cO(n)\to \cO'(n)$ such that $f_1(\id_{\cO}) = \id_{\cO'}$, and such that the diagram
    \[
        \begin{tikzcd}
            \cO(n)\times\cO(k_1)\times\dots\times \cO(k_n) \ar[d,"f_n\times f_{k_1}\times \dots \times f_{k_n}"'] \ar[r,"\gamma"] & \cO(k_+) \ar[d,"f_{k_+}"] \\
            \cO'(n)\times\cO'(k_1)\times\dots\times \cO'(k_n) \ar[r,"\gamma'"] & \cO'(k_+)
        \end{tikzcd}
    \]
    commutes for all choices of $n$ and $k_1,\dots, k_n$. 
\end{definition}

The most basic and fundamental example of an operad arises by collecting all $n$-ary operations on a fixed object $X$. Concretely, these are just the morphisms $X^n \to X$ in the ambient category, organized into the following operad.

\begin{example}
    Fix an object $X \in \cat{C}$ and define $\mathcal{E}nd_X(m) \coloneqq \cat{C}(X^m, X)$. 
    The collection of sets $\mathcal{E}nd_X(m)$ for $m \geq 0$, equipped with the natural composition maps and $\Sigma_n$-action given by permuting the inputs, forms an operad in $(\Set, \times, \ast)$ called the \textit{endomorphism operad} of $X$.
\end{example}

In this way, an operad should be thought of as parametrizing abstract $n$-ary operations. 
The notion of an algebra over an operad then formalizes how such abstractly parametrized operations can act concretely on a given object.

\begin{definition}[{\cite[Lemma 1.4]{May:Geometry}}]
    An \emph{algebra over an operad $\cO$} is an object $X\in \cat{C}$ together with maps
    \[
        \mu_n\colon \cO(n)\times X^{n}\to X
    \]
    that are associative, unital, and compatible with the $\Sigma_n$-actions.
    Equivalently, an $\cO$-algebra structure on $X$ is specified by a morphism of operads from $\cO$ to the endomorphism operad $\mathcal{E}nd_X$.
\end{definition}

An $\cO$-algebra structure on $X$ realizes the operations parametrized by $\cO$ as operations on $X$. More precisely, for any $n\geq 0$ and any $\alpha\in \cO(n)$, the structure map gives rise to an $n$-ary operation $X^n\to X$ given by $x\mapsto \mu_n(\alpha,x)$. The compatibility conditions ensure that these operations assemble together in the way specified by the operad.

A central example of an operad in topology is the case of \emph{$E_\infty$-operads} in spaces, which encode operations that are associative and commutative up to all higher homotopies. 

\begin{definition}
    An operad $\mathcal{E}$ in $\Top$ is called an \emph{$E_\infty$-operad} if for each $n \geq 0$, the space $\mathcal{E}(n)$ is contractible and has a free action of the symmetric group $\Sigma_n$. That is, $\mathcal{E}(n)$ is a universal space for the group $\Sigma_n$.
\end{definition}

These operads describe algebraic structures on topological spaces whose operations are unique up to a contractible space of choices (``up to all coherent homotopies"). We next present three specific examples of $E_\infty$-operads that play a prominent role below: the linear isometries operad, the little disks operad, and the Steiner operad. 

\subsubsection*{Linear isometries operad}
Let $\mathcal{J}$ be the category whose objects are finite or countably infinite dimensional real inner product spaces and whose morphisms are linear isometric embeddings. A linear isometric embedding is a linear transformation that preserves the inner product and is an isomorphism onto its image. Since we are working in infinite dimensions, such a map must be injective but need not be an isomorphism. Giving the mapping spaces the compact-open topology, $\mathcal{J}$ becomes a topological category.

Let $\RR^\infty\coloneqq \bigoplus_{n \geq 0} \RR$ be the real vector space of $\mathbb{N}$-indexed vectors with finite support. We turn this into an inner product space with the standard inner product 
\(
    \langle x, y \rangle \coloneqq \sum_{i = 0}^\infty x_i y_i,
\)
and topologize it with the induced norm.

\begin{definition}[Linear isometries operad, {\cite[Definition 1.2]{MayEringBook}}]\label{def: linear isometry operad}
    The \emph{linear isometries operad} $\mathcal{L}$ is the operad in $(\Top,\times,\ast)$ with $n$-th space
    \[
        \mathcal{L}(n) = \mathcal{J}((\RR^{\infty})^n,\RR^\infty).
    \]
    The structure maps are given by
    \begin{enumerate}
    	\item for $f\in \mathcal{L}(n)$ and $g_i\in \mathcal{L}(k_i)$ we have
    	 \[
            \gamma(f;g_1,\dots,g_n) 
                = 
            f\circ(g_1\oplus\dots\oplus g_n)\in \mathcal{L}(k_+), \text{ and}
        \]
    	 \item for $\sigma\in \Sigma_n$ and $x_i\in \mathbb{R}^{\infty}$ we have
    	 \[
    	 	(f \cdot \sigma)(x_1,\dots,x_n) = f(x_{\sigma(1)},x_{\sigma^{-1}(2)},\dots,x_{\sigma(n)}).
    	 \]
    \end{enumerate}
    The unit element $\id \in \mathcal{L}(1)$ is the identity map on $\RR^{\infty}$.
\end{definition}

The linear isometries operad is important because it provides a convenient $E_{\infty}$-operad for constructing algebras in model categories of spectra \cite{EKMM}.  The structure of $\mathcal{L}$ leads to the existence of various power operations which can be used for computations in algebraic topology.  For a longer discussion on these operations, see \cite{Lawson:operations}.

\subsubsection*{Little disks operad}
Let $D^k$ be an open unit disk in $\mathbb{R}^k$ and $\mathcal{D}_k(n)$ be the space of $n$-tuples $(f_1,\dots,f_n)$ where the maps $f_i \colon D^k \to D^k$ are orientation-preserving affine embeddings with pairwise disjoint images. The space $\mathcal{D}_k(n)$ has a $\Sigma_n$-action by permuting the tuples. Observe that there is a map $\mathcal{D}_k(n) \to \mathcal{D}_{k+1}(n)$ induced from the equatorial embedding $D^k \hookrightarrow D^{k+1}$.

\begin{definition}[Little disks operad]
    The \textit{little disks operad} $\mathcal{D}$ is the operad in $(\Top, \times, \ast)$ with $n$-th space 
    \[
        \mathcal{D}(n) \coloneqq \colim_k \mathcal{D}_k(n).
    \]
    The unit of this operad is the element of the colimit represented by the identity morphisms $\id \in \mathcal{D}_k(1)$. The operad structure maps come from composition of maps in the expected way. 
\end{definition}

For any values of $k$ and $n$ let $\mathrm{Conf}_{k}(n)$ denote the space of configurations of $n$-points in $\mathbb{R}^k$, topologized as a subspace of $\mathbb{R}^{nk}$.  The map $\mathcal{D}_{k}(n)\to \mathrm{Conf}_k(n)$ which sends a tuples of maps $(f_1,\dots,f_n)\in \mathcal{D}_k(n)$ to $(f_1(0),\dots,f_n(0))$ is a homotopy equivalence.  The map $\mathrm{Conf}_k(n)\to \mathrm{Conf}_{k}(n-1)$ which forgets the last embedded point is a fibration, and an inductive argument using the Serre spectral sequence shows that $\mathrm{Conf}_k(n)$ is $k-2$ connected for any $n$.  Thus, taking the colimit over $k$ we see that $\mathcal{D}(n)$ is contractible for all $n$, hence the little disks operad $\mathcal{D}$ is an $E_\infty$-operad.

\subsubsection*{Steiner operad}

We now turn our attention to the Steiner operad. Let $V$ be a finite dimensional inner product space and let $R_V$ denote the space of distance reducing embeddings $\alpha \colon V \to V$.
A Steiner path is a continuous map $h\colon [0,1] \to R_V$ such that $h(1) = \id_V$. Let $P_V$ be the space of Steiner paths with the compact open topology.
We write $\mathcal{K}_V(j)$ for the subspace of $P_V^j$ consisting of all tuples $(h_1,\dots,h_j)$ such that the images of $h_i(0)$ are pairwise disjoint. This space has an action of the symmetric group $\Sigma_j$ by permutation.

\begin{figure}[hbt!]
\begin{tikzpicture}
	\draw[thick] (0,0) -- (10,0);
	\foreach \x in {0,2,4,...,10}{
		\draw[thick] (\x,0.5) -- (\x,5);
		\draw[thick] (\x-0.1,0.5) -- (\x+0.1,0.5);
		\draw[thick] (\x-0.1,5) -- (\x+0.1,5);
		\draw[thick] (\x,-0.1) -- (\x,0.1) 
			node[below=0.2]{
				\ifthenelse{\x=0}{$0$}{
					\ifthenelse{\x=10}{$1$}{
						$0.\x$
					}
				}
			};
		}
    \fill[red!10]
      (0,1)
        .. controls (0.7,1)   and (1.3,0.75) .. (2,0.75)
        .. controls (2.8,0.75) and (3.3,1.5)  .. (3.95,1.5)
        .. controls (4.8,1.5)  and (5.3,3)    .. (6,3)
        .. controls (6.8,3)    and (7.4,2)    .. (7.95,2)
        .. controls (8.6,2)    and (9.3,0.5)  .. (9.95,0.5)
        -- (9.95,5)
        .. controls (9.3,5)    and (8.6,4.75) .. (7.95,4.75)
        .. controls (7.4,4.75) and (6.8,4.75) .. (6,4.75)
        .. controls (5.3,4.75) and (4.8,4)    .. (3.95,4)
        .. controls (3.3,4)    and (2.8,3.25) .. (2,3.25)
        .. controls (1.3,3.25) and (0.7,3)    .. (0,3)
        -- cycle;

    \fill[blue!15, fill opacity=0.5]
      (0,3.75)
        .. controls (0.7,3.75) and (1.3,3.5)  .. (2,3.5)
        .. controls (2.8,3.5)  and (3.5,3)    .. (4.05,3)
        .. controls (4.8,3)    and (5.3,1.5)  .. (6,1.5)
        .. controls (6.8,1.5)  and (7.5,1)    .. (8.05,1)
        .. controls (8.8,1)    and (9.4,0.5)  .. (10.05,0.5)
        -- (10.05,5)
        .. controls (9.4,5)    and (8.8,3.5)  .. (8.05,3.5)
        .. controls (7.5,3.5)  and (6.8,2.75) .. (6,2.75)
        .. controls (5.3,2.75) and (4.5,4.25) .. (4.05,4.25)
        .. controls (3.4,4.25) and (2.7,4.5)  .. (2,4.5)
        .. controls (1.3,4.5)  and (0.7,4.75) .. (0,4.75)
        -- cycle;
      
	\foreach \x/\ymin/\ymax in {0/1/3, 
								2/0.75/3.25, 
								3.95/1.5/4, 
								6/3/4.75, 
								7.95/2/4.75, 
								9.95/0.5/5}
		{
		\draw[ultra thick,red] (\x,\ymin) -- (\x,\ymax);
		\draw[thick,red] (\x-0.1,\ymin) -- (\x+0.1,\ymin);
		\draw[thick,red] (\x-0.1,\ymax) -- (\x+0.1,\ymax);
		}
	\foreach \x/\ymin/\ymax in {0/3.75/4.75, 
								2/3.5/4.5,
								4.05/3/4.25, 
								6/1.5/2.75, 
								8.05/1/3.5, 
								10.05/0.5/5}
		{
		\draw[ultra thick,blue] (\x,\ymin) -- (\x,\ymax);
		\draw[thick,blue](\x-0.1,\ymin) -- (\x+0.1,\ymin);
		\draw[thick,blue] (\x-0.1,\ymax) -- (\x+0.1,\ymax);
		}
	\node[red] at (-0.5,2) {$h_1(0)$};
	\node[blue] at (-0.5,4.25) {$h_2(0)$};
\end{tikzpicture}
\caption{A heuristic illustration of an element of $\mathcal{K}_V(2)$.}
\end{figure}
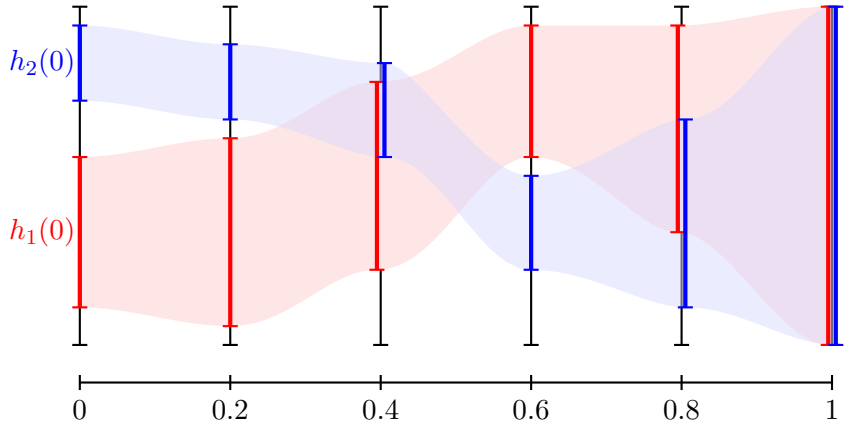

\begin{definition}[Steiner operad, \cite{Steiner}]\label{def: steiner operad}
	The \emph{Steiner operad} is the operad in $(\Top,\times,\ast)$ whose $n$-th space $\mathcal K(n)$ is defined by
    \[\mathcal K(n) \coloneqq \colim\limits_{V\subset \RR^\infty} K_V(n)\]
    where the colimit runs over all finite dimensional subspaces of $\RR^\infty$. 
    The constant Steiner path at the identity embedding in $\mathcal K(1)$ is the identity element. Abusing notation, we denote this constant Steiner path at the identity by $\id \in \mathcal K(1)$. Composition is done pointwise on the paths; given 
	\begin{align*}
		g = (g_1,\ldots,g_n) &\in \mathcal K(n),\\
		f_i = (f_{i,1},\ldots,f_{i,k_i}) &\in \mathcal K(k_i), 
        \quad i = 1,\ldots,n 
	\end{align*}
    the composite $g \circ (f_1,\ldots,f_n)$ is:
	\[
        \big(
            g_1 \circ f_{1,1}, g_1 \circ f_{1,2},\ldots, g_1 \circ f_{1,k_1}, 
            g_2 \circ f_{2,1},                   \ldots, g_2 \circ f_{2,k_2},
                \ldots, 
            g_n \circ f_{n,1},                   \ldots, g_n \circ f_{n,k_n}) 
	\]
    where $g_{s} \circ f_{s,j}$ denotes pointwise composition. 
	In words, we put the pointwise composites $g_s \circ f_{s,j}$ in a tuple ordered lexicographically: first by $s$ then by $j$. 
\end{definition}

The Steiner operad should be viewed as a ``thickening'' of the little disks operad whose elements are pairwise disjoint tuples of distance reducing embeddings $\mathbb{R}^{\infty}\to \mathbb{R}^{\infty}$.  The Steiner operad was introduced in \cite{Steiner} to correct a technical problem which arose in \cite{MayEringBook} when considering pairings between the linear isometries operad and the little disks operad, see \cref{subsec:operad_pairings} for the definition of pairings of operads. 

\subsection{\texorpdfstring{$N_\infty$}{N-infinity}-operads}\label{subsec:background_Ninfinity}
As said before, the linear isometries and Steiner operads are two examples of $E_\infty$-operads. In the presence of a $G$-action, the appropriate generalization of $E_\infty$-operads is the notion of $N_\infty$-operads of Blumberg and Hill \cite{BlumbergHillOperadic}. These objects were created to capture additional structure present when working with operads in $G$-spaces and algebras in $G$-spaces or $G$-spectra.

Let us now illustrate such potential additional structure with an analogous example. Given a finite group $G$ acting on a vector space $V$, we have not only addition maps $V^{\oplus n} \to V$, but also canonical transfer maps $V \to V^G$ given by $v\mapsto \sum_{g\in G} g\cdot v$. Note that the commutativity of addition is necessary for the output to be a $G$-fixed point. In fact, for any chain of subgroups $H\leq K\leq G$ there is a similar map $V^H\to V^K$ obtained by summing over the action by representatives of the $H$-cosets in $K$.  The $N_{\infty}$-operads are meant to encode operations that are commutative up to all higher homotopies and certain transfer maps. Unlike in the example here, these transfers are no longer automatic because commutativity doesn't hold strictly.

\begin{remark}\label{remark:gsigma}
    Let $\cO(n)$ be an operad in $(\Top^G, \times, \ast)$. Since $\Sigma_n$ acts on $\cO(n)$ on the right through $G$-equivariant maps, we can think of $\cO(n)$ as a left $G\times \Sigma_n$-space, where the left $\Sigma_n$-action comes from the right inverse action. That is, for any $(g,\sigma)\in G\times \Sigma_n$ and $f\in\cO(n)$, the action is defined by 
    \[
        (g,\sigma)\cdot f \coloneqq gf\sigma^{-1}.
    \] 
    The identity $\id \in \cO(1)$ is a $G$-fixed point. 
\end{remark}

\begin{definition}[{\cite[Definition 3.7]{BlumbergHillOperadic}}]\label{defn: Ninfinity operads}
    An \emph{$N_{\infty}$-operad} $\cO$ is a reduced operad in $(\Top^G, \times, \ast)$  such that
    \begin{enumerate}
        \item the action of $\Sigma_n$ on $\cO(n)$ is free for all $n\geq 0$;
        \item for all subgroups $\Gamma \leq G \times \Sigma_n$, the fixed points $\cO(n)^\Gamma$ are either empty or contractible;
        \item for all $n\geq 0$, the fixed point space $\cO(n)^G$ is non-empty.
    \end{enumerate}      
\end{definition}

\begin{remark} 
An $N_{\infty}$-operad encodes a commutative (up to coherent homotopies) operation between fixed point spaces of algebras, as we now explain.  For simplicity we let $G=C_2$ be the cyclic group with $2$ elements and let $\cO$ be an $N_{\infty}$-operad.  If $\gamma$ generates $C_2$ and $\sigma$ generates $\Sigma_2$ then let $\Gamma = \{1,(\gamma,\sigma)\}\leq C_2\times \Sigma_2$.  

If $X$ is an $\cO$-algebra in $C_2$-spaces, then there is a $C_2\times \Sigma_2$-equivariant map $\mu_2\colon \cO(2)\times X^2\to X$, where the $\Sigma_2$-action on $X$ is trivial. One can verify there are homeomorphisms $(\cO(2)\times X^2)^{\Gamma}\cong \cO(2)^\Gamma \times X$, and $X^{\Gamma}\cong X^{C_2}$, and thus, taking $\Gamma$-fixed points yields a map
    \[
        \cO(2)^{\Gamma}\times X\to X^{C_2}.
    \]
    If $\cO(2)^{\Gamma}$ is empty this contains no information, but if  $\cO(2)^{\Gamma}$ is contractible then this provides a homotopy coherent ``summing over the orbits'' operation, as described at the beginning of this subsection.

    The construction of operations $X^K\to X^H$ for $K \leq H \leq G$ is similar in the case of a general finite group $G$. Let $n = |H/K|$.  The left action of $H$ on $H/K$ determines a (well-defined up to conjugacy) homomorphism $\alpha\colon H\to \Sigma_{n}$ and we write $\Gamma\leq G\times \Sigma_{n}$ for the subgroup of elements of the form $(h,\alpha(h))$ for any $h\in H$.  The fixed points $\cO(n)^{\Gamma}$ determine whether or not the operad $\cO$ is indexing summing operations $X^K\to X^H$ on any $\cO$-algebra $X$. For a more detailed discussion see \cite[\S 7]{BlumbergHillOperadic}.
\end{remark}

\begin{remark} 
    Let $\cO$ be an $N_{\infty}$-operad.  For any $n$, suppose that $\Gamma\leq G\times \Sigma_n$ is any subgroup containing an element of the form $(1,\sigma)$ with $\sigma\neq 1$.  Since $\cO(n)$ has a free $\Sigma_n$-action, it follows that $\cO(n)^\Gamma=\emptyset$, as no element is fixed by $(1,\sigma)$.  Thus, condition (b) in \cref{defn: Ninfinity operads} need only be checked for subgroups $\Gamma\leq G\times \Sigma_n$ such that $\Gamma\cap (1\times \Sigma_n) = \{(1,1)\}$.  One can check that every such subgroup is the set of elements of the form $(k,\alpha(k))$ for some $K\leq G$ and group homomorphism $\alpha\colon K\to \Sigma_n$. 
\end{remark}

This justifies the following definition.

\begin{definition}\label{def:graph_subgroup}
        Let $H$ be a subgroup of $G$, and $X$ be an $H$-set of size $n$. Without loss of generality, we may assume $X = \underline{n}$. Denote by $\varphi$ the induced group homomorphism $\varphi\colon H\to \Sigma_n$. The \emph{graph subgroup} of $X$ is
\[
    \Gamma(X) = \{(h,\varphi(h))\in G\times \Sigma_n\mid h\in H\}.
\]
\end{definition}

\begin{remark}
    Let $\cO$ be an $N_{\infty}$-operad. 
    For any $n\geq 0$ and $H\leq G$ we have $\cO(n)^G\leq \cO(n)^H$, 
    so $\cO(n)^H$ is non-empty and therefore must be contractible.  
    It is straightforward to check that the collection $\cO^H(n) \coloneqq \cO(n)^H$ 
    forms an operad in $\Top$ in which all constituent spaces are contractible and 
    $\Sigma_n$-free. 
    In particular, these operads are all $E_{\infty}$. 
    So if $X$ is an $\cO$-algebra, then $X^H$ is an $E_\infty$ space for each $H \leq G$.
\end{remark}

The $N_\infty$-operads that arise most often in equivariant algebraic topology are equivariant versions of the linear isometries operad, the little disks operad, and the Steiner operad. For all of these definitions we need the notion of a $G$-universe. 

\begin{definition}
   A \emph{$G$-universe} $U$ is a countably infinite-dimensional orthogonal real $G$-representation 
   containing the trivial $1$-dimensional real $G$-representation, 
   such that if $V \subset U$ is a finite-dimensional subrepresentation 
   then $U$ contains infinitely many copies of $V$. 
\end{definition}

\begin{example}
    A \textit{complete $G$-universe} $U$ is any $G$-universe isomorphic to an infinite direct sum of copies of the regular representation $\rho$, that is, 
    \(
        U \cong \bigoplus_{n = 0}^\infty \rho. 
    \)  
    This universe is complete in the sense that every finite-dimensional irreducible $G$-representation appears infinitely often as a summand.
\end{example}

\begin{example}
    Let $G = C_2$, the cyclic group of order $2$. There are two irreducible real $C_2$-representations: the trivial representation $\RR$ and the sign representation $\sigma$. There are two $C_2$-universes: the trivial universe $\RR^\infty$ and the complete universe $(\RR \oplus \sigma)^\infty$. 
\end{example}

\subsubsection*{Equivariant linear isometries} 
Equivariant generalization is simple once we have the language of universes. 

\begin{definition}
    Let $G$ be a finite group and let $U$ be a $G$-universe. 
 The \textit{$G$-equivariant linear isometries operad} $\mathcal{L}_G(U)$ 
 is the $G$-operad in $(\Top^G, \times, \ast)$ with $n$-th space 
    \[
        \mathcal{L}_G(U)(n) \coloneqq \mathcal{J}(U^{\oplus n}, U), 
    \]
    with $G$-action given by $g \cdot f = gfg^{-1}$ and $\Sigma_n$-action given by permuting the inputs. Composition maps and the identity element are as in the (non-equivariant) linear isometries operad (\cref{def: linear isometry operad}). 
\end{definition}

The equivariant linear isometries operad is 
an $N_\infty$-operad by~\cite[Lemma 3.15]{BlumbergHillOperadic}. 

\subsubsection*{Equivariant little disks} 
If $V$ is a finite-dimensional orthogonal $G$-representation, let $D(V) \coloneqq \{v \in V \mid \|v\| < 1\}$ denote the open unit disk in $V$.
Let $\mathcal{D}_V(n)$ be the space of $n$-tuples $(f_1,\dots,f_n)$ 
where the maps $f_i \colon D(V) \to D(V)$ are affine (not necessarily equivariant) 
injective maps with pairwise disjoint images. The space $\mathcal{D}_V(n)$ has 
a $G$-action by conjugation on each component map, namely $g \cdot f_i = gf_ig^{-1}$, 
and a $\Sigma_n$-action by permuting the tuples. 

If $V \subseteq W$, there are maps $\varphi_n\colon \mathcal{D}_V(n) \to \mathcal{D}_W(n)$ 
defined as follows.  For any $w\in \mathcal{D}(W)$, there is a unique decomposition 
$w = v+x$ where $v\in D(V)\subseteq D(W)$ and $x\in V^{\perp}$ and let $\varphi_n(f)(w) = f(v)+x$. 
It is straightforward to check that $\varphi_n(f)$ is an affine embedding $D(W)\to D(W)$ 
and $\varphi_1$ preserves the identity map. 

\begin{definition}[Equivariant little disks]
    Let $U$ be a $G$-universe. The \textit{equivariant little disks operad} $\mathcal{D}(U)$ is the operad in $(\Top^G, \times, \ast)$ with $n$-th space 
    \[
        \mathcal{D}(U)(n) \coloneqq \colim \mathcal{D}_V(n),
    \]
    where the colimit is taken over all finite-dimensional subrepresentations $V \subset U$.

    The unit of this operad is the element of the colimit represented by the identity morphisms $\id \in \mathcal{D}_V(1)$. The operad structure maps come from composition of maps in the expected way. 
\end{definition}

The equivariant little disks operad $\mathcal{D}(U)$ is an $N_\infty$-operad for all choices of universe $U$ by~\cite[Corollary 3.14]{BlumbergHillOperadic}. 

\subsubsection*{Equivariant Steiner operad}
Let $U$ be a $G$-universe and $V\subset U$ a finite dimensional subrepresentation. Let $R_V$ denote the space of (non-equivariant) distance reducing embeddings $\alpha \colon V \to V$ which are. Note that $R_V$ is a $G$-space with $G$-action given by conjugation: $g \cdot \alpha = g\alpha g^{-1}$.

Recall that a Steiner path is a continuous map $h \colon [0,1] \to R_V$ such that $h(1) = \id_V$. Let $P_V$ denote the space of all Steiner paths; this is again a $G$-space with $G$-action coming from the $G$-action on $R_V$. 

As before, we write $\mathcal{K}_V(j)$ for the subspace of $P_V^j$ consisting of all tuples $(h_1, \ldots, h_j)$ such that the images of $h_j(0)$ are disjoint. This space has a $G$-action coming from the $G$-action on $P_V$ and a $\Sigma_j$-action by permuting tuples. 

\begin{definition}[Equivariant Steiner operad]
    The \emph{equivariant Steiner operad} associated to the $G$-universe $U$ is the operad in $(\Top^G, \times, \ast)$ whose $n$-th space $\mathcal{K}_U(n)$ is defined by
    \[ \mathcal{K}_U(n)\coloneqq \colim\limits_{V \subset U} K_V(n)\]
    where the colimit runs over all finite dimensional subrepresentation of $U$.
    The identity element $\id \in \mathcal{K}_U(1)$ and the composition are defined as in the (non-equivariant) Steiner operad \ref{def: steiner operad}.
\end{definition}

The equivariant Steiner operad is an $N_\infty$-operad 
by~\cite[Corollary 3.14]{BlumbergHillOperadic}. 

\begin{remark}
    The equivariant Steiner operad $\mathcal{K}(U)$ is equivariantly equivalent to the equivariant little disks operad $\mathcal{D}(U)$, just as with their non-equivariant versions, see~\cite[Proposition 3.13]{BlumbergHillOperadic}. Classically, the (non-equivariant) linear isometries operad $\mathcal{L}$ and the (non-equivariant) little disks operad $\mathcal{D}$ are also homotopy equivalent. But for all but three finite groups, there are universes $U$ such that $\mathcal{L}(U)$ and $\mathcal{D}(U)$ are inequivalent, see~\cite[Theorem 4.22]{BlumbergHillOperadic}. 
\end{remark}

\section{Indexing and transfer systems}\label{sec:indexing_transfer_systems}

Indexing and transfer systems\textemdash first introduced in \cite{BlumbergHillOperadic} and \cite{BalchinBarnesRoitzheim,RubinCombinatorial}, respectively\textemdash are combinatorial gadgets that encode homotopy types of $N_\infty$-operads. More concretely, as we recall below in \cref{theorem: indexing systems are operads}, 
the homotopy category of $N_\infty$-operads, the poset of indexing systems, and the poset of transfer systems, are all equivalent. 
These equivalences allow us to pass from homotopy theory to combinatorics, where the classification of transfer systems yields a classification of homotopy types of $N_\infty$-operads. This interplay has revealed many new and interesting connections between equivariant homotopy theory and combinatorics, see \cite{BlumbergHillOsornoOrmsbyRoitzheim} for a thorough introduction to the subject. 

Although indexing and transfer systems for a group $G$ are in bijection (\cref{thm: indexing and transfer systems are the same}) there are often advantages to working with one over the other. Transfer systems consist of a finite amount of data, making them well suited to combinatorial arguments. On the other hand, sometimes it is important to work with finite $G$-sets instead of subgroups of $G$, and in these situations indexing systems are more natural.

\subsection{Indexing Systems}
First, we fix some notation. For $G$ a finite group with subgroups $K\leq H\leq G$, there is an adjunction
\[
\begin{tikzcd}[column sep=huge]
	{\Set^K} & {\Set^H}
	\arrow[""{name=0, anchor=center, inner sep=0}, "{\Ind^H_K}", bend left=30, from=1-1, to=1-2]
	\arrow[""{name=1, anchor=center, inner sep=0}, "{\Res^H_K}", bend left=30, from=1-2, to=1-1]
	\arrow["\dashv"{anchor=center, rotate=-90}, draw=none, from=0, to=1]
\end{tikzcd}
\]
with left adjoint (\emph{induction})
\[
    \Ind^H_K(T) \coloneqq H\times_KT
\]
and right adjoint (\emph{restriction}) $\Res^H_K$ given by restricting the $H$-action to a $K$-action.  
    For any $g \in G$, the \textit{conjugation functor} at $g$ is the functor
    \[
        c_{g,H}\colon \Set^H\to \Set^{gHg^{-1}}
    \]
    that sends an $H$-set $X$ to the $gHg^{-1}$-set $gX \coloneqq \{gx\mid x\in X\}$.

\begin{notation} 
    We write $\Fin^G$ for the full subcategory of $\Set^G$ spanned by the finite $G$-sets.
\end{notation}

\begin{definition}[{\cite[Definition 3.22]{BlumbergHillOperadic}}]\label{definition:indexing_systems}
    An \emph{indexing system} $\cI$ for a finite group $G$ consists of a collection $\{\cI(H)\}_{H\leq G}$ 
   of full subcategories $\cI(H)\subseteq \Fin^H$ satisfying the following. 
   \begin{enumerate}
       \item For all $H\leq G$ the category $\cI(H)$ is closed under isomorphisms, disjoint unions, subobjects, and finite products. 
       \item The restriction and conjugation functors $\Res^H_K$ and $c_{g,H}$ restrict to functors between the categories $\{\cI(H)\}_{H\leq G}$. 
       \item (closure under self induction) If $H/K$ is in $\cI(H)$ and $X \in \cI(K)$ then $\Ind^H_K(X)\in \cI(H)$. 
   \end{enumerate}
\end{definition}

	The collection of indexing systems for a group $G$ forms a poset where $\cI_1\leq \cI_2$ if and only if $\cI_1(H)\subseteq \cI_2(H)$ for all $H$. We denote this poset by $\mathrm{Index}(G)$.

\begin{example}\label{example: example of indexing systems 1}
	Every non-trivial finite group $G$ admits at least two indexing systems:
	\begin{enumerate}
		\item\label{indexing system 1} the \emph{complete indexing system} $\mathcal{I}_c$ with $\mathcal{I}_c(H) = \Fin^H$ for all $H$, and
		\item the \emph{trivial indexing system} $\cI_0$, where $\cI_0(H)$ is the full subcategory of $\Fin^H$ consisting of trivial $H$-sets.  
	\end{enumerate}
	These indexing systems are, respectively, the maximal and minimal elements in the poset of indexing systems for any finite group $G$.
\end{example}

\begin{example}\label{example: example of indexing systems 2}
	Let $G = C_{p^2}$, the cyclic group with $p^2$ elements for some prime $p$.  We define
	\begin{align*}
		\cI(C_{p^2}) &=\langle C_{p^2}/e, C_{p^2}/C_{p^2}\rangle\\
		\cI(C_{p}) &=\Fin^{C_p}\\
		\cI(e) &= \Fin
	\end{align*}
	where the notation $\langle C_{p^2}/e, C_{p^2}/C_{p^2}\rangle$ should be read as the full subcategory of $\Fin^{C_{p^2}}$ which is generated under disjoint unions by the orbits $C_{p^2}/e$ and $C_{p^2}/C_{p^2}$. Explicitly, this consist of all $C_{p^2}$-sets which have no elements with stabilizer exactly $C_{p}$.  We leave it to the reader to carefully check this is an indexing system, and that it is distinct from both the complete indexing system and the trivial one.
\end{example}

\begin{remark}
    For any subgroup $H \le G$, each category $\cI(H)$ is non-empty: it contains all finite sets with trivial $H$-action. Indeed, any such set $X$ can be expressed as a finite disjoint union of singletons, so closure under disjoint unions reduces the problem to the case of a singleton. But a singleton is the terminal object in $\Fin^H$, and hence belongs to $\cI(H)$ by closure under finite products.
\end{remark}

Condition \ref{indexing system 1} in the definition of an indexing system $\cI$ is fairly restrictive.  Since every finite $H$-set is the union of its orbits, and $\cI(H)$ is closed under disjoint unions and subobjects, we see that $\cI(H)$ is entirely determined by which transitive $H$-sets it contains.  All such $H$-sets have the form $H/K$ for some subgroup $K\leq H$.  This gives a tractable way of writing down an indexing system, since we need only specify a collection of transitive $H$-sets for all $H$ and then check that the induction, restriction, and conjugation functors behave correctly on this finite set of objects.

\subsection{Transfer systems}
The above observation naturally leads to the notion of transfer systems, which record\textemdash via relations among subgroups\textemdash which such orbits are allowed under self-induction, while respecting conjugation and restriction.

\begin{definition}
    A \emph{transfer system} for a group $G$ is a poset relation $\to$ on $\mathrm{Sub}(G)$ satisfying the following conditions. 
    \begin{enumerate}
        \item It refines inclusion: if $K\to H$ then $K\leq H$. 
        \item It is closed under conjugation: if $K\to H$ then $gKg^{-1}\to gHg^{-1}$ for all $g\in G$. 
        \item It is closed under restriction: if $L\leq G$ and $K\to H$ then $K\cap L\to H\cap L$. 
    \end{enumerate}
\end{definition}

The collection of transfer systems is a poset ordered by refinement. Transfer systems provide a convenient and tractable way to encode the data of an indexing system using relations among subgroups of $G$.  In particular, we have the following. 

\begin{theorem}[{\cite{BalchinBarnesRoitzheim,RubinDetecting}}]\label{thm: indexing and transfer systems are the same}
There is a poset isomorphism between the poset of indexing systems and the poset of transfer systems for a finite group G.
\end{theorem}

Under this correspondence, an indexing system $\cI$ determines a transfer system where $K\to H$ if and only if $H/K\in \cI(H)$.  Conversely, given a transfer system $\to$ we build an indexing system by declaring it to be the smallest indexing system which contains every orbit $H/K$ for $K\to H$.

\begin{example}\label{example:basic_transfer_systems}Every non-trivial finite group $G$ admits at least two transfer systems:
	\begin{enumerate}
		\item the \emph{complete transfer system} $\to_c$ given by the subset relation, and 
		\item the \emph{trivial transfer system} $\to_0$, which is equality; that is, $K\to_0 H$ if and only if $K=H$.  
	\end{enumerate}
	These two transfer systems correspond, under \cref{thm: indexing and transfer systems are the same}, to the indexing systems of \cref{example: example of indexing systems 1}
\end{example}

\begin{example}
    Let $G = C_{p^2}$, the cyclic group with $p^2$ elements for some prime $p$.  There is a transfer system give by the poset relation given by all the equalities, $e\to C_{p^2}$, and $e\to C_p$.  We can represent this transfer system pictorially as follows. 
    
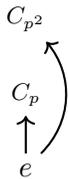
\[
    	\begin{tikzpicture}
				\node (1) at (0,0) {$e$};
				\node (Cp) at (0,1) {\footnotesize$C_p$};
				\node (Cp2) at (0,2) {\footnotesize$C_{p^2}$};
				\draw[->, thick] (1) -- (Cp);
				\draw[->, thick] (1) edge[bend right=1.5cm] (Cp2);
		\end{tikzpicture}
    \]
    This transfer system corresponds, under \cref{thm: indexing and transfer systems are the same}, to the indexing system of \cref{example: example of indexing systems 2}.
\end{example}

\begin{example}[{\cite[Example 4.9]{RubinDetecting}}]
    The $\Sigma_3$-transfer systems 
    \[
        \sthreetriv
        \qquad
        \sthreeab
        \qquad
        \sthreec
        \qquad
        \sthreeall
    \]
    are the $\Sigma_3$-transfer systems that are realized by equivariant Steiner operads. Here, $H_1, H_2, H_3$ are the three conjugate copies of $C_2$ inside $\Sigma_3$.  
\end{example}

\begin{remark}
    \label{remark saturated transfer systems}
    The class of transfer systems satisfying the two-out-of-three property is called \textit{saturated}. These transfer systems are of particular interest because of their relation to linear isometries operads. In fact, the saturated transfer systems can also be characterized as those transfer systems which are self-compatible in the sense of \cref{def compatible pair transfer system}, see \cite[Corollary 7.77]{BlumbergHillBiincomplete}. 
\end{remark}

For the purposes of this paper, we adopt the perspective of indexing systems as the primary framework, while keeping the correspondence with transfer systems in mind.

\subsection{\texorpdfstring{$N_{\infty}$}{N-infinity}-operads and indexing systems}
The relationship between $N_{\infty}$-operads and indexing systems is established by considering the fixed points of $\mathcal{O}(n)$ with respect to the graph subgroups of $G\times \Sigma_n$, as in \cref{def:graph_subgroup}. 

\begin{definition}
    Let $\cO$ be an $N_\infty$-operad. We say that an $H$-set $X$ is \emph{$\cO$-admissible} if the fixed point space $\cO(n)^{\Gamma(X)}$ is contractible for all $n \geq 0$. 
\end{definition}

Note that while the graph subgroup $\Gamma(X)$ depends on the choice of bijection $X \cong \underline{n}$, any two choices will define conjugate subgroups in $G\times \Sigma_n$ and thus the fixed points $\cO(n)^{\Gamma(X)}$ do not depend on this choice, at least up to homeomorphism.

\begin{proposition}[{\cite[Section 4]{BlumbergHillOperadic}}]\label{proposition: admissible sets are indexing system}
    Let $\cO$ be an $N_{\infty}$-operad and let $\cI_{\cO}(H)\subset \Set^H$ denote the full subcategory of $\cO$-admissible $H$-sets.  Then the collection $\{\cI_{\cO}(H)\}$ is an indexing system.
\end{proposition}

Up to a suitable notion of weak equivalence, $N_{\infty}$-operads are classified by their associated indexing system.  Explicitly, we say that a map of $N_{\infty}$-operads $f : \cO_1\to \cO_2$ is a weak equivalence if for all $n$ the map $f_n : \cO_1(n)\to \cO_2(n)$ gives a weak homotopy equivalence after taking fixed points with respect to all subgroups of $G\times \Sigma_n$.    The \emph{homotopy category of $N_{\infty}$-operads}, which we denote by $\mathrm{Ho}(N_{\infty})$, is the categorical localization of the category of $N_{\infty}$-operads at the collection of weak equivalences. 

\begin{theorem}[{\cite{BlumbergHillOperadic,BonventrePereira,GutierrezWhite,RubinCombinatorial}}]\label{theorem: indexing systems are operads}
    The assignment $\cO\mapsto \cI_{\cO}$ descends to an equivalence of categories $\mathrm{Ho}(N_{\infty}) \simeq \mathrm{Index}(G)$. 
\end{theorem}

\section{Pairings of operads and compatiblity of indexing systems}

It is often the case that algebraic structures have two different operations\textemdash one additive and one multiplicative\textemdash that are compatible in a manner analogous to addition and multiplication in a ring. A key piece of this compatibility is a version of the  distributive law. Here, we are concerned with the compatibility of two operads which encode multiplicative and additive structure, and compatibility of the indexing systems that correspond to a pair of such $N_\infty$-operads. The archetypal example of this is the distributive law of multiplication over addition in $\mathbb{N}$; in order to discuss compatibility of operads, we first describe a categorification of this ordinary distributive law in \cref{subsection:distributivity bijection}. In what remains of this section, we recall a ``distributive law" for operads from \cite{MayEringBook} (\cref{subsec:operad_pairings}) and for indexing systems from \cite{BlumbergHillBiincomplete} (\cref{subsec:pairs_indexing_systems}). 

\subsection{Categorifying distributivity} 
\label{subsection:distributivity bijection}

To discuss operad pairings, we must first introduce a categorification of the distributive property of multiplication over addition and set some notation. 

Let $k, j_1,\dots j_k\geq 0$. Given natural numbers $i_{r,q}$ with $1\leq r \leq k$ and $1\leq q \leq j_r$, the most general form of the distributive law is an equality 
\begin{equation}
\label{equation:distributitivity}
    \prod_{r=1}^k \sum_{q=1}^{j_r} i_{r,q} 
    = 
    \sum_{\bm{q}} \prod_{r=1}^k i_{r,q_r},
\end{equation}
where the sum on the right is taken over all tuples $\bm{q} = (q_1, \ldots, q_k) \in \underline{j_1} \times\cdots \times  \underline{j_k}$, ordered lexicographically. 

Given two totally ordered sets $A$ and $B$, we order their disjoint union $A\amalg B$ so that $a \leq b$ for all $a \in A$ and all $b \in B$, while retaining the internal orderings on $A$ and $B$. With this convention, \cref{equation:distributitivity} categorifies into a bijection
\[
    \nu \colon \prod_{r=1}^k \underline{i_{r,+}}
    \xlongrightarrow{\cong}
    \coprod_{\bm{q}}
    \prod_{r=1}^k \underline{i_{r,q_r}}, 
\]
where $i_{r,+} = \sum_{q=1}^{j_r} i_{r,q}$. 

The abstract existence of such a bijection is not enough for our purposes; our goal is to make $\nu$ explicit. The example below illustrates the general strategy. 

\begin{example}
\label{ex:3x5=15}
     Basic arithmetic gives us
\[
    (2+3)\times (1+2+3) = (2\times 1)+(2\times 2)+(2\times 3)+(3\times 1)+(3\times 2)+(3\times 3)
\]
and this defines an order-preserving bijection between the two sets
\[
    \nu\colon \underline{5}\times \underline{6} \xrightarrow{\cong} (\underline{2}\times \underline{1}) \amalg (\underline{2}\times \underline{2}) \amalg (\underline{2}\times \underline{3}) \amalg (\underline{3}\times \underline{1}) \amalg (\underline{3}\times \underline{2}) \amalg (\underline{3}\times \underline{3})
\]
which, for instance, sends the element $(5,2)\in \underline{5}\times \underline{6}$ to the element $\nu(5,2) = (3,1)\in \underline{3}\times \underline{2}$. 
\end{example}

To make $\nu$ explicit, let $\bm{\ell} = (\ell_1, \ldots, \ell_k) \in \prod_{r=1}^k \underline{i_{r,+}}$ with $\ell_r \in \underline{i_{r,+}}$. For each $r$, 
there is a unique $\gamma_r(\bm{\ell})$ such that $1 \leq \gamma_r(\bm{\ell}) \leq j_r$ and 
\[
    \sum\limits_{q=1}^{\gamma_r(\bm{\ell})-1} i_{r,q}
    < \ell_r \leq 
    \sum\limits_{q=1}^{\gamma_r(\bm{\ell})} i_{r,q}.
\]
In other words, $\gamma_r(\bm{\ell})$ is the integer such that $\ell_r$ is coming from the $i_{r,\gamma_r(\bm{\ell})}$ block of the decomposition
\[
    \underline{i_{r,+}}
    \cong 
    \coprod\limits_{q=1}^{j_r}\underline{i_{r,q}}.
\]
 Define $\beta_r(\bm{\ell})$ by
\[
    \beta_r(\bm{\ell}) \coloneqq \left(\ell_r - \sum\limits_{q=1}^{\gamma_r(\bm{\ell})-1}i_{r,q}\right)\in \underline{i_{r,\gamma_r(\bm{\ell})}}.
\]
The integers $\beta_r(\bm{\ell})$ tell us the element $\ell_r$ corresponds to in $\underline{i_{r,\gamma(\ell_r)}}$. Finally, let $\bm{q}(\bm{\ell}) = (\gamma_1(\bm{\ell}),\dots,\gamma_k(\bm{\ell}))\in \underline{j_1}\times\dots\times \underline{j_k}$. The bijection $\nu$ is then given by
\[  
    \nu(\bm{\ell}) = 
    (\beta_1(\bm{\ell}),\dots,\beta_k(\bm{\ell}))
    \in \prod\limits_{r=1}^k \underline{i_{r,\gamma_r(\bm{\ell})}}
    \xhookrightarrow{\bm{q} = \bm{q}(\bm{\ell})} \coprod\limits_{\bm{q} \in \underline{j_1} \times\cdots \times  \underline{j_k}}\ 
    \prod\limits_{r=1}^k \underline{i_{r,q_r}}. 
\]

\begin{example}
    In \cref{ex:3x5=15} above, we have $\bm{\ell} = (5,2)$ so that $\ell_1=5$ and $\ell_2=2$.  We have $\gamma_1(\bm{\ell}) = 2 $ and $\gamma_2(\bm{\ell}) = 2$ so $\bm{q}(\bm{\ell})= (2,2)$.  The sequences $\bm{q}$ are ordered according to the lexicographical ordering on $\underline{2}\times \underline{3}$ so this is the fifth component of the disjoint union. 
 We have $\beta_1(\bm{\ell}) = 3$ and $\beta_2(\bm{\ell}) = 1$ so $\nu(5,2) = (3,1)\in \underline{3}\times \underline{2}$ as claimed.
\end{example}

\subsection{Pairings of operads}\label{subsec:operad_pairings}
While operads index operations on a space $X$, it is often the case that $X$ might have two distinct flavors of operation.  By analogy with algebra, it is helpful to think about one of these operations as a multiplication and the other as an addition, so that $X$ can be thought of as a ring up to homotopy.  To keep track of this data using operads, we need to use one operad $\cP$ for the multiplication and another $\cQ$ for the addition. Just like in algebra, where the multiplication and addition are related by the distributive law, these two operads are related by a specified structure which is called a \emph{pairing}.  

In this section we recall May's definition of a pairing of operads in \cite{MayEringPaper} and give the fundamental examples of the linear isometries and Steiner operads.  In the next section give a definition of compatible pairs of indexing systems. The relationship between pairings of $G$-operads and compatible pairs of indexing systems occupies the final three sections of the paper.

Before we discuss operad pairings, we first introduce some notation, as well as recall the notation $k_+$ from \cref{notation:block permutation}. 

\begin{notation} 
\label{notation:permutation in blocks}
Let $k_1, \ldots, k_n$ be non-negative integers and let $k_\times = \prod_{i=1}^n k_i$. We identify the finite set $\underline{k_\times}$ with $\underline{k_1} \times  \cdots \times \underline{k_n}$ via the order-preserving bijection that imposes the  lexicographical ordering on tuples $\bm{q} \in \underline{k_1} \times \cdots \times \underline{k_n}$.
	\begin{enumerate}[(a)]
		\item Let $\sigma \in \Sigma_n$. We write 
		\[
			\sigma_\times(k_1, \dots, k_n) \in \Sigma_{k_\times}
		\]
		for the permutation obtained by writing 
		\(
			\underline{k_\times} \cong \underline{k_1} \times  \cdots \times \underline{k_n}
		\) 
        via the order-preserving bijection described above
		and permuting the factors in the product according to $\sigma$.

		\item Let $\tau_i \in \Sigma_{k_i}$. We write 
		\[
			\tau_1 \otimes  \dots \otimes \tau_n \in \Sigma_{k_\times}
		\]
		for the permutation of $\underline{k_\times}$ obtained by writing
		\(
			\underline{k_\times} \cong \underline{k_1} \times  \dots \times \underline{k_n}
		\) 
        via the order-preserving bijection described above
		and acting on $\underline{k_i}$ with $\tau_i$. 
	\end{enumerate}
\end{notation}

Fix two reduced operads $\cP$ and $\cQ$.  We write $\mathrm{id}_{\cP}$ and $\id_{\cQ}$ for the identity elements in  $\cP(1)$ and $\cQ(1)$, respectively. 
The reader should keep in mind that we intend for $\cP$ to act as a multiplicative operad and $\cQ$ to acts as an additive operad.  We will write
\[
	\mu\colon \cP(k)\times \cP(i_1)\times\dots\times \cP(i_k)\to \cP(i_+)
\]
and 
\[
    \alpha\colon \cQ(k)\times \cQ(i_1)\times\dots\times \cQ(i_k)\to \cQ(i_+)
\]
for the corresponding operad structure maps.  We write $p$, with various subscripts, for a generic element in the operad $\cP$. We will write $x$, with various subscripts, for a generic element in the operad $\cQ$. In particular, given $k\geq 0$, $j_r\geq 0 $ for all $1\leq r \leq k$, and $i_{r,q}\geq 0 $ for all $1\leq r \leq k$ and $1\leq q \leq j_r$, let
\begin{align*}
  p       &\in \cP(k)       &x       &\in \cQ(j)         \\
  p_r     &\in \cP(j_r)     &x_r     &\in \cQ(j_r)       \\
          &                 &x_{r,q} &\in \cQ(i_{r,q})
\end{align*} 
where $1\leq r\leq k$, $1\leq q\leq j_r$, and $i_{r,q}$ are some non-negative integers for all $1\leq r\leq k$ and all $1\leq q\leq j_r$. To simplify notation in what is to come, write 
\[
    x_{J_r} = (x_{r,1},\ldots,x_{r,j_r}) \in \cQ(i_{r,1}) \times \cdots \times \cQ(i_{r,j_r}). 
\]
and 
\[
    x_{\bm{q}} = (x_{1,q_1},x_{2,q_2},\dots, x_{k,q_k})\in \cQ(i_{1,q_1})\times \dots\times \cQ(i_{k,q_k}),
\]
where $\bm{q} = (q_1,\ldots,q_k) \in \underline{j_1} \times \cdots \times \underline{j_k}$.  
	
\begin{definition}[{\cite[Definition VI.1.6]{MayEringBook}}]
\label{def:operad_pairing} 
    An \emph{action of $\cP$ on $\cQ$} is a collection of maps
	\[
        \lambda\colon \cP(k)\times \cQ(j_1)\times\dots\times \cQ(j_k)\to \cQ(j_\times)
	\]
    satisfying the following conditions:
      \begin{enumerate}[label=(\alph*)]
      \setlength{\itemsep}{8pt}
          \item\label{def:operad_pairing_first_condition} $\lambda(\mu(p;p_1,\dots,p_k);x_{J_1},\dots,x_{J_k}) = \lambda(p;\lambda(p_1;x_{J_1}),\dots,\lambda(p_k; x_{J_k}))$, 
          \item\label{def:operad_pairing_second_condition} $\alpha(\lambda(p;x_1,\dots, x_k); \prod\limits_{\bm{q}} \lambda(p;x_{\bm{q}}))\circ \nu = \lambda(p; \alpha(x_1; x_{J_1}),\dots,\alpha(x_k;x_{J_k}))$, where $\nu$ is the \textit{distributivity bijection} of \cref{subsection:distributivity bijection}, 
          \item\label{action-id-P} $\lambda(\mathrm{id}_{\cP};x) =x$,
          \item $\lambda(p;\mathrm{id}_{\cQ},\dots,\mathrm{id}_{\cQ}) =\mathrm{id}_{\cQ}$,
          \item\label{action-sigma} $\lambda(p \cdot \sigma; x_1,\dots,x_k) = \lambda(p;x_{\sigma^{-1}(1)},\dots,x_{\sigma^{-1}(k)})\cdot \sigma_{\times}(j_1,\dots,j_k)$,
          \item $\lambda(p; x_1 \cdot \tau_1,\dots, x_k \cdot \tau_k) = \lambda(p;x_1,\dots,x_k)\cdot (\tau_1\otimes \dots \otimes \tau_k)$,
          \item $\lambda(\ast;)=\id_{\cQ}$ when $k=0$,
          \item\label{action-when-xi-ast} $\lambda(p;x_1,\dots, x_k)=\ast$ if for any $i$, $x_i=\ast \in \cQ(0)$\footnote{Note that because our operads are reduced, this last condition is superfluous.}.
      \end{enumerate}
When there is an action of $\cP$ on $\cQ$, we say that the triple $(\cP,\cQ, \lambda)$ is a  \emph{pairing of operads} or an \emph{operad pairing}, and we use the terms interchangeably. Moreover, when there is no ambiguity, we will use simply $(\cP,\cQ)$. 
\end{definition}

We can represent the equations \ref{def:operad_pairing_first_condition} and  \ref{def:operad_pairing_second_condition} diagramatically as follows:

\begin{equation}
    \tag*{\ref{def:operad_pairing_first_condition}}
\begin{tikzcd}
	{\cP(k)\times\prod\limits_{r=1}^k\left(\cP(j_r)\times\prod\limits_{q=1}^{j_r}\cQ(i_{r,q})\right)} & {\cP(k)\times\prod\limits_{r=1}^k\cQ(i_{r,\times})} \\
	{\left(\cP(k)\times\prod\limits_{r=1}^k\cP(j_r)\right)\times\prod\limits_{r=1}^k\prod\limits_{q=1}^{j_r}\cQ(i_{r,q})} \\
	{\cP(j_+)\times\prod\limits_{r=1}^k\prod\limits_{q=1}^{j_r}\cQ(i_{r,q})} & {\cQ(n)}
	\arrow["{1\times (\prod \lambda)}", from=1-1, to=1-2]
	\arrow["\cong"', from=1-1, to=2-1]
	\arrow["\lambda", from=1-2, to=3-2]
	\arrow["{\mu\times 1}"', from=2-1, to=3-1]
	\arrow["\lambda"', from=3-1, to=3-2]
\end{tikzcd}
\end{equation} 

\begin{equation}
    \tag*{\ref{def:operad_pairing_second_condition}}
\begin{tikzcd}
	\displaystyle{\cP(k)\times \prod_{r=1}^k \left(\cQ(j_r)\times\prod_{q=1}^{j_r}\cQ(i_{r,q}) \right)} & {\cP(k)\times \prod\limits_{r=1}^k \cQ(i_{r,+})} \\
	{\cP(k)^{j_\times+1}\times \prod\limits_{r=1}^k \left(\cQ(j_r)\times\prod\limits_{q=1}^{j_r}\cQ(i_{r,q})^{j_{\widehat{r}}} \right)} & {\cQ(\prod\limits_{r=1}^{k}i_{r,+})}\\
	{\left(\cP(k)\times \prod\limits_{r=1}^k \cQ(j_r)\right) \times\left(\prod\limits_{\bm{q}}\cP(k)\times\prod\limits_{r=1}^{k}\cQ(i_{r,q_r})\right)} & {\cQ(\prod\limits_{r=1}^{k}i_{r,+})} \\
	{\cQ(j_\times)\times\prod\limits_{\bm{q}}\cQ(\prod\limits_{r=1}^ki_{r,q_r})} & {\cQ(\sum\limits_{\bm{q}}\prod\limits_{r=1}^{k}i_{r,q_r})} 
	\arrow["{1\times\left(\prod\alpha\right)}", from=1-1, to=1-2]
	\arrow["{\Delta\times (\prod(1\times \left(\prod \Delta\right)))}"', from=1-1, to=2-1]
	\arrow["\lambda", from=1-2, to=2-2]
	\arrow["\cong"', from=2-1, to=3-1]
	\arrow["{\lambda\times \left(\prod \lambda\right)}", from=3-1, to=4-1]
	\arrow["\alpha"', from=4-1, to=4-2]
	\arrow["{=}"', from=4-2, to=3-2]
	\arrow["\nu"', from=3-2, to=2-2]
\end{tikzcd}
\end{equation}
where $n = \prod_{r=1}^k\prod_{q= 1}^{j_r}i_{r,q}$ 
and $j_{\widehat{r}} = j_1\cdots \widehat{j}_{r}\cdots j_k$. 

\begin{remark}\label{rema:operad-construction-linear-steiner}
    As mentioned in the introduction, while all (non-equivariant) $E_{\infty}$-operads are homotopy equivalent, constructions of pairings depend quite heavily on the specific operads being chosen. For instance, while there is a canonical pairing of the linear isometries and Steiner operads, it is not known whether or not the Steiner operad admits a pairing with itself.
\end{remark}

\begin{remark}
    While we will not explore this idea further in this paper, we pause to explain the relationship between a pairing of operads $(\cP,\cQ)$ and the algebras over $\cP$ and $\cQ$.  To make this discussion precise, recall that any operad $\cO$ in spaces determines an associated monad $\mathbb{O}$ on the category of spaces so that $\cO$-algebras are the same thing as $\mathbb{O}$-algebras.  It is shown in \cite[\S 4, Appendix A]{MayEringPaper} that the data of a pairing $(\cP,\cQ)$ is equivalent to the statement that the monad $\mathbb{P}$ on the category of spaces lifts to a monad on the category of $\mathbb{Q}$-algebras.  This amounts to checking that the structure maps of the monad $\mathbb{P}$ can be made maps of $\mathbb{Q}$-algebras when necessary, and this is the data of the pairing map. This perspective is in line with the view of the operad pairing as encoding a distributive law: as May notes, the relationship between the monads $\mathbb{P}$ and $\mathbb{Q}$ is exactly that of a distributive law between monads as defined by Beck \cite{Beck}.
\end{remark}

We recall an example of an operad pairing between the linear isometries operad and the Steiner operad that will be essential for us.  In \cref{cor:linearisometries_Steiner_pair_intersection_monoids} we provide a new proof of this fact that sheds light into the construction of the pairing maps.

\begin{proposition}[{\cite[Section 3]{MayEringPaper} and \cite{Steiner}}]
\label{prop:operad_pairing_linearisometries_Steiner} 
There is an action of the linear isometries operad $\cL$ on the Steiner operad $\cK$. 
\end{proposition}

\begin{proof}[Proof sketch.] 
We outline the construction of the action of the linear isometries operad on the Steiner operad. Suppose we have an element $f\in \cL(n)$, and an $n$-tuple $\bm{e}=\left(e_1, \ldots, e_n\right)$, where each $e_i\colon\RR^\infty\to \RR^\infty$ is a linear embedding. We can construct a new embedding, which we will denote by $f\ast \bm{e}$, as the composite
\begin{center}
\begin{tikzcd}
	\RR^\infty\arrow[r,"="] & \im(f)\oplus \im(f)^\perp\ar[d,"f^{-1}\oplus \id"] & {} & \im(f)\oplus \im(f)^\perp \arrow[r,"="] & \RR^\infty.\\ 
	{} & (\RR^\infty)^n\oplus \im(f)^\perp\arrow[rr,"(\bigoplus_{i=1}^n e_i)\oplus \id"] & {} & (\RR^\infty)^n\oplus \im(f)^\perp \arrow[u,"f\oplus \id"] & {}
\end{tikzcd}
\end{center}

Note that if the $e_i$ are distance reducing for all $i\in \underline{n}$, 
then so is the map $f\ast \bm{e}$, and we can extend this action to 
Steiner paths in the natural way. We can get an action 
\[ 
	\lambda\colon \cL(n)\times\cK(j_1)\times\cdots\times\cK(j_n)\longrightarrow \cK(j_{\times})
\]
by sending $(f;h_1, \dots h_n)$, where each $h_i$ is itself a $j_i$-tuple of Steiner paths, to the $j_1\cdots j_n$-tuple given by $f\ast (h_{1,q_1},\dots , h_{n,q_n})$, where $(q_1,\dots,q_n) \in \underline{j_1} \times \cdots \times \underline{j_n}$ is ordered lexicographically. This defines an operad pairing. 
\end{proof}

\subsection{Pairings of indexing systems}\label{subsec:pairs_indexing_systems}
Given a pairing $(\cP, \cQ)$ of $N_\infty$-operads, it is reasonable to ask what additional structure is present on the associated indexing systems. In this section, we recall the definition of compatible pairs of indexing systems following \cite{BlumbergHillBiincomplete}. A comparison of operad pairings and compatible pairs of indexing systems is given in \cref{section: operads to indexing systems}. 

\begin{definition}
\label{def: indexing system pair}
	A pair of $G$-indexing systems $(\cI_{m},\cI_a)$ is a \emph{compatible pair} if for all subgroups $K\leq H\leq G$ such that $H/K\in \cI_m(H)$ and all $X\in \cI_{a}(K)$ we have that the coinduced $H$-set $\Set^K(H,X)$ is in $\cI_{a}(H)$.
\end{definition} 

The first author gives an equivalent formulation of compatibility of $N_\infty$-data in the form of a compatible pair of transfer systems \cite[Definition 4.6]{ChanBiincomplete}, that we recall below. 

\begin{definition} 
\label{def compatible pair transfer system}
A pair $(\mathcal{T}_1,\mathcal{T}_2)$ of transfer systems is a \emph{compatible pair} if $\mathcal{T}_1\leq \mathcal{T}_2$ and for all subgroups $H,K,L \leq G$ such that $K \leq H$ and $L \leq H$, 
if $K \to H$ in $\mathcal{T}_1$ and $K \cap L \to K$ in $\mathcal{T}_2$, then $L \to H$ in $\mathcal{T}_2$. In other words, given the solid arrows in the diagram below, with each arrow in the indicated transfer system, there exists a dashed arrow as below.
    \[
        \begin{tikzcd}
                & 
            H 
                \\
            K
                \ar[ur, "\mathcal{T}_1"] 
                & 
                & 
            L 
                \ar[ul, dashed, "\mathcal{T}_2"']
                \\
                &   
            K \cap L 
                \ar[ul,"\mathcal{T}_2"]
                \ar[ur, "\mathcal{T}_1"']
        \end{tikzcd}
    \]
\end{definition}

\begin{remark}
	In \cite{BlumbergHillBiincomplete}, Blumberg--Hill actually define compatible pairs of \emph{indexing categories}, which are an equivalent formulation of the data of indexing systems.  Unwinding the equivalence of indexing systems and indexing categories, one can see that the definition above is indeed equivalent to the data of a compatible pair of indexing categories given in \cite[Theorem 7.65]{BlumbergHillBiincomplete}.
\end{remark}

\begin{remark}
    Blumberg--Hill's original interest in compatible pairs of indexing categories is related to their work on the algebraic structures known as \emph{bi-incomplete Tambara functors}. These serve as analogs of commutative rings in equivariant homotopy theory.  Explicitly, a multiplicative equivariant cohomology theory is naturally valued in (graded) bi-incomplete Tambara functors. We will not give the full definition here, as we will not use bi-incomplete Tambara functor explicitly in this paper, but refer the reader to \cite{BlumbergHillBiincomplete} for a lengthier discussion. 
\end{remark}

The definition of compatible pairs of indexing systems  can be understood algebraically through the following example.  For any group $H$ let $A(H)$ denote the Burnside ring of $H$. The elements of $H$ are virtual finite $H$-sets, addition is induced by taking disjoint unions of $H$-sets, and multiplication is induced by taking Cartesian products.  If $K\leq H$ then coinduction $\Set^K(H,-)\colon \Set^K\to \Set^H$ induces a function $\mathrm{nm}_K^H\colon A(K)\to A(H)$ which is multiplicative, though not additive.
	
If $\cI$ is a $G$-indexing system and $H\leq G$ one might also consider the \emph{partial Burnside ring} $A_{\cI}(H)\subset A(H)$ which is generated by those $H$-sets belonging to $\cI(H)$. If $K\leq H\leq G$, it is not true in general that the operations $\mathrm{nm}_K^H$ give functions on the partial Burnside rings, since their image might not be contained in $A_{\cI}(K)$. The compatibility condition on a pair $(\cI_m,\cI_a)$ precisely guarantees that $\mathrm{nm}_K^H\colon A_{\cI}(K)\to A_{\cI}(H)$ is well defined for all $H/K\in \cI_m(H)$.
We end this section with some examples and a non-example.

\begin{example}
	If $\cI_{a}$ is the complete indexing system where $\cI_{a}(H) \coloneqq \Fin^H$ for all $H$, then $(\cI_{m},\cI_a)$ is compatible for any $\cI_m$.  Indeed, $\Set^K(H,X)\in \cI_a(H)$ for any $X\in \Set^H$ and any $K \leq H$.
\end{example}

\begin{example}
	If $\cI_m$ is the trivial indexing system, where $\cI_m(H)$ consists of only trivial $H$-sets, then $(\cI_m,\cI_a)$ is compatible for any choice of $\cI_a$. Here the condition $H/K\in \cI_m(H)$ only holds when $H=K$, in which case for any choice of $X\in \cI_a(H)$ we have $\Set^H(H,X)\cong X\in \cI_a(H)$.
\end{example}

\begin{example}[Non-example]
	Let $G = C_{4}$ and let $\cI_a = \cI_{m} = \cI$ be the indexing system from \cref{example: example of indexing systems 2} in the case $p=2$.  We claim that the pair $(\cI,\cI)$ is \emph{not} compatible.  To see this, note that $C_{4}/e\in \cI(C_{4})$ and $\underline{2}\in \cI(e)$, but 
	\[
		\Set^e(C_{4},\underline{2})\cong  (C_4/C_4)^{\amalg 2}\amalg (C_4/C_2)\amalg (C_4/e)^{\amalg 3} 	
	\]
	which is not in $\cI(C_4)$ since there is an orbit isomorphic to $C_4/C_2$.
\end{example}

\section{Compatible indexing systems from pairings of \texorpdfstring{$N_\infty$}{N-infinity}-operads}
\label{section: operads to indexing systems}

As we have explained, our goal is to understand the relation between pairings of $N_\infty$-operads and the corresponding pair of associated indexing systems. The goal of this section is to establish a precise connection in one direction. Specifically, this section consists of a proof of the following.

\begin{theorem}
\label{prop:operad-action-induces-compatible-pair}
 Let $\cP$ and $\cQ$ be $N_\infty$-operads with indexing systems $\cI_\cP$ and $\cI_\cQ$, respectively. If there is an action of $\cP$ on $\cQ$, then $(\cI_\cP,\cI_\cQ)$ is a compatible pair of indexing systems.
\end{theorem}
To prove this theorem, we need to check the conditions of \cref{def: indexing system pair}.  Specifically, given subgroups $K\leq H\leq G$ such that $H/K\in \cI_{P}(H)$ and a $K$-set $X\in \cI_{Q}(K)$ we need to show that the coinduced $H$-set $\Set^K(H,K)$ is in $\cI_Q(H)$. We prove this below, after setting notation and proving some preliminary results.

Recall from \cref{remark:gsigma} that if $\cO$ is an $N_{\infty}$-operad, we will treat left $G$-action and the right $\Sigma_n$-action on $\cO(n)$ as a combined  left $G \times \Sigma_n$-action. Explicitly, if $x\in \mathcal{O}(n)$ and $(g,\sigma)\in G\times \Sigma_n$ we have $(g,\sigma)\cdot x = gx\sigma^{-1}$.

Let $K\leq H\leq G$ be a chain of subgroups, and let $n = |H/K| = [H:K]$. We fix coset representatives $h_1 K, \dots ,h_n K$ of $H/K$, and consider the homomorphism that the action of $H$ on $H/K$ determines,
\begin{equation}
\label{eq:defining homomorphism}
    \sigma\colon H\to \Sigma_n,
\end{equation}
characterized by $h \cdot h_i K = h_{\sigma(h)(i)}K$ for all $i = 1, \ldots, n$. This homomorphism $\sigma$ also defines a graph subgroup of $H/K$ (relative to this choice of coset representatives) by 
\[
	\Gamma(H/K) \coloneqq \big\{ (h, \sigma(h)) \mid h \in H \big\} \leq G \times \Sigma_n
\]
as in \cref{def:graph_subgroup}.  We will need the following technical lemma.

\begin{lemma}\label{eq:def-of-kih}
Let $G$ be a finite group, and consider a chain of subgroups $K\leq H\leq G$ with a fixed choice of coset representatives $h_1,\dots,h_n$. Then for all $i \in \underline{n}$ and all $h \in H$, there is a unique $k = k_{i,h} \in K$ such that 
	\[
		h_i^{-1}h = k_{i,h}h_{\sigma(h)^{-1}(i)}^{-1}.
	\]
\end{lemma}

\begin{proof}
For all $i$ we have 
\[
    Kh_i^{-1}h = (h^{-1}h_iK)^{-1} = (h_{\sigma(h)^{-1}(i)}K)^{-1} = Kh_{\sigma(h)^{-1}(i)}^{-1}.
\]
Therefore, $h_i^{-1}h$ differs from $h^{-1}_{\sigma(h)^{-1}(i)}$ by left multiplication by a unique element of $K$. 
\end{proof}

Let $X$ be a $K$-set with $|X| = m$. Without loss of generality, we may assume that $X = \underline{m}$ and the $K$-action on $X$
is given by a group homomorphism 
\[
    \tau\colon K\to \Sigma_m\hspace{0.5em}\text{ via  }\hspace{0.5em} k\cdot x \coloneqq \tau(k)(x) \hspace{0.5em}\text{ for all }\hspace{0.5em}x\in X.
\]
\cref{eq:def-of-kih} and the morphism $\sigma$ of \eqref{eq:defining homomorphism} 
allow us to define a left $H$-action on $\underline{m}^{\underline{n}}$ via 
\[
	h \cdot (a_1, \ldots, a_n) = 
	(\tau(k_{1,h})(a_{\sigma(h)^{-1}(1)}), \dots, \tau(k_{1,h})(a_{\sigma(h)^{-1}(n)})) 
\] 
Observe that this determines and is determined by a homomorphism
\begin{equation}
\label{eq:h action on coinduced set}
\begin{tikzcd}[row sep=0]
	\theta \colon H \ar[r] & \Sigma_{m^n}\\
		\phantom{\nu \colon} h \ar[r, mapsto] & \big(\tau(k_{1,h})\otimes\dots\otimes \tau(k_{n,h})\big)\circ \sigma(h)_\times(m,\dots, m),
\end{tikzcd}
\end{equation}
using \cref{notation:permutation in blocks}.
The next lemma gives a very explicit model for the coinduced $H$-set $\Set^K(H,X)$, 
which we use in the proof of \cref{prop:operad-action-induces-compatible-pair}. 

\begin{lemma}\label{lemma: explicit action for coinduced sets}
    There is an $H$-equivariant bijection
    \[
        \Set^K(H,X)\cong \underline{m}^{\underline{n}},
    \]
    where $\Set^K(H,X)$ is an $H$-set via $(h \cdot f)(h') = f(h'h)$ and $\underline{m}^{\underline{n}}$ is an $H$-set via $\theta$. 
\end{lemma}

\begin{proof}
    We have the natural bijections
    \begin{align*}
        \Set^K(H,X) 
        	&\cong 
			\Set^K\left(\coprod_{i = 1}^n{K h_i^{-1}}, X\right) 
			\cong 
			\Set\left(\coprod_{i = 1}^n\{h_i^{-1}\}, X\right)
				\\
            &\cong 
            \prod_{i = 1}^n{\Set(\{h_i^{-1}\}, X)} 
            \cong 
            \prod_{i = 1}^n{X} 
            \cong \underline{m}^{\underline{n}}.
    \end{align*}
    Given $f\in \Set^K(H,X)$, the above bijection is given by
    \[
        f\mapsto (f(h_{1}^{-1}),\dots,f(h_n^{-1})).
    \]
    Thus
    \begin{equation*}
        (h\cdot f) \mapsto (f(h_1^{-1}h), \dots, f(h_n^{-1}h)).
    \end{equation*}
    For all $h\in H$ and $i=1\dots, n$ we have
    \[
        f(h_i^{-1}h) = f(k_{i,h}h^{-1}_{\sigma(h)^{-1}(i)}) = k_{i,h}\cdot f(h^{-1}_{\sigma(h)^{-1}(i)}) = \tau(k_{i,h})(f(h^{-1}_{\sigma(h)^{-1}(i)}))
    \]
    where the first equality holds by \cref{eq:def-of-kih}, the second because $f\in \Set^K(H,X)$, and the third by definition. It follows that the bijection sends $(h\cdot f)$ to the tuple 
    \[
        (\tau(k_{1,h})\otimes\dots\otimes \tau(k_{n,h}))\cdot(f(h^{-1}_{\sigma(h)^{-1}(1)}),\dots,f(h^{-1}_{\sigma(h)^{-1}(n)})).
    \]
    This is equal to
    \[
       \bigg( (\tau(k_{1,h})\otimes\dots\otimes \tau(k_{n,h}))\circ\sigma(h)_\times( m,\dots, m) \bigg) \cdot \big(f(h_{1}^{-1}),\dots,f(h_n^{-1})\big)
    \]
    which is the action of $\theta(h)$, as claimed.
\end{proof}

\begin{proof}[Proof of \cref{prop:operad-action-induces-compatible-pair}]\label{proposition: proof of theorem 5.1}
Following \cref{def: indexing system pair}, we must show that given $H/K \in \cI_\cP(H)$ with $n = |H \colon K|$, and $X \in \cI_\cQ(K)$ with $m = |X|$, then $\Set^K(H,X)$ is in $\cI_\cQ(H)$. 

Let us consider $\theta$ as in \eqref{eq:h action on coinduced set}. To show that $\Set^K(H,X) \in \cI_{\cQ}(K)$, by \cref{lemma: explicit action for coinduced sets} it is enough to show that 
\[
	\mathcal{Q}(m^n)^{\Gamma_\theta} = \mathcal{Q}(|\Set^K(H,X))|)^{\Gamma_\theta}
\] 
is non-empty, where
\[
	\Gamma_\theta = \big\{ (h,\theta(h)) \mid h\in H\big\}.
\]
Let $\lambda$ be the action of $\cP$ on $\cQ$. The idea of the proof is to produce elements $a\in \cI_{\cP}(n)$ and $b_1,\dots,b_n\in \cQ(m)$ such that $\lambda(a;b_1,\dots,b_n)\in \cQ(m^n)$ is a $\Gamma_{\theta}$-fixed point.  We first define the elements $a$ and $b_1,\dots,b_n$, and then check that they produce the desired fixed point. 

Since $H/K \in \cI_\cP(H)$ then $\cP(n)^{\Gamma(H/K)}$ is non-empty, so we may choose $a \in \mathcal{P}(n)^{\Gamma(H/K)}$. In particular, for all $h \in H$, 
\begin{equation}\label{eq:a-in-P-fixed}
(h,\sigma(h))\cdot a = a.
\end{equation}
Since $X \in \cI_\cQ(K)$ then $\cQ(m)^{\Gamma(X)}$ is non-empty, so we may choose $b \in \cQ(m)^{\Gamma(X)}$. 
In particular, for all $k \in K$,
\[
(k,\tau(k)) \cdot b = b. 
\]
This is equivalent to 
\begin{equation}
\label{eq:Gamma(X) fixed point relation}
	(e,\tau(k)) \cdot b = (k^{-1},e) \cdot b
\end{equation}
for all $k\in K$. We define $b_i \coloneqq (h_i,e)\cdot b$ for $i = 1, \dots, n$, and observe that for all $h \in H$, 
\begin{align*}
(h^{-1},e)(h_i,e)\cdot b &= (h^{-1} h_i,e)\cdot b \\
&= (h_{\sigma(h)^{-1}(i)} k_{i,h}^{-1},e)\cdot b \\
 &= (h_{\sigma(h)^{-1}(i)},e) (k_{i,h}^{-1},e)\cdot b \\
 &= (h_{\sigma(h)^{-1}(i)},e) (e,\tau(k_{i,h})) \cdot b \\
 &= (e,\tau(k_{i,h})) (h_{\sigma(h)^{-1}(i)},e) \cdot b \\
 &= (e,\tau(k_{i,h})) \cdot b_{\sigma(h)^{-1}(i)},
\end{align*}
where 
the second equality follows from \cref{eq:def-of-kih}, the fourth equality follows from \eqref{eq:Gamma(X) fixed point relation}, and the sixth equality is the definition of $b_{\sigma(h)^{-1}(i)}$.  
Multiplying both sides by $(h, e)$ we see that, for any $h\in H$,
\begin{equation}\label{eq:htau-action-bi}
b_i = (h_i,e)\cdot b = (h,\tau(k_{i,h})) \cdot b_{\sigma(h)^{-1}(i)}.
\end{equation}

We claim that $z \coloneqq \lambda(a; b_1, \dots, b_n) \in \mathcal{Q}(m^n)^{\Gamma_\theta}$. Let $h \in H$, so by definition 
\[
(h, \theta(h))\cdot z 
	= \big(h, (\tau(k_{1,h})\otimes\dots\otimes \tau(k_{n,h}))\circ \sigma(h)_\times( m,\dots, m) \big)\cdot \lambda(a; b_1, \dots, b_n). 
\]
Because the action $\lambda$ is compatible with the $G$-action and the $\Sigma_i$-actions, this becomes 
\[
 \lambda((h,\sigma(h))\cdot a; (h,\tau(k_{1,h}))\cdot b_{\sigma(h)^{-1}(1)}, \dots, (h,\tau(k_{n,h}))\cdot b_{\sigma(h)^{-1}(n)}),
\]
which is equal to $z = \lambda(a; b_{1}, \dots, b_{n})$ by \eqref{eq:a-in-P-fixed} and \eqref{eq:htau-action-bi}. Therefore, $\cQ(m^n)^{\Gamma_\theta}$ is non-empty, so $\Set^K(H,X) \in \cI_\cQ(H)$.  
\end{proof}

\section{Building operads from monoids}

In this section we focus on a method for defining operads and pairings of operads from the simpler structure of monoids. While this is interesting in its own right, our motivation for presenting it here lies in \cref{sec:Ninfinity_from_indexing}, where it allows us to build examples of indexing systems from monoids. In \cref{subsec:operads-from-monoids} we begin by recalling the main construction in \cite{Giraudo-operads-from-monoids} that associates a symmetric operad to any monoid, and expand this study to operad pairings. In \cref{subsec:operads-from-intersection-monoids} we refine the previous results using \emph{intersection} monoids, originally introduced in \cite{szczesny2025realizingtransfersystemssuboperads}, which allows for a tighter control and the construction of operad pairings\textemdash this feature is crucial for our objective of realizing pairings of transfer systems. 

\subsection{Operads from monoids}\label{subsec:operads-from-monoids} We build on a construction of Giraudo \cite[Section 2]{Giraudo-operads-from-monoids}, which in turn builds on ideas of Boardman--Vogt \cite{BoardmanVogt} and Igusa \cite{Igusa}, that associates a symmetric operad to a monoid.
Recall \cref{notation:block permutation,notation:permutation in blocks}.

\begin{definition}\label{def:operad_from_M}
    Let $M$ be a monoid. We define its associated symmetric operad $\cO(M)$ in $(\Set, \times, \ast)$ by
    \[
    	\mathcal{O}(M)(n)\coloneqq M^n,
	\]
    with right $\Sigma_n$-action on $\cO(M)(n)$ given by permuting the tuples. Explicitly, for $x=(x_1,\dots,x_n)\in\cO(M)(n)$ and $\sigma\in\Sigma_n$, we have 
\[
    x \cdot \sigma \coloneqq (x_{\sigma^{-1}(1)},\dots,x_{\sigma^{-1}(n)}).
\]
    Note that $\mathcal{O}(M)(0)=\ast$. The composition rule in this operad is defined as follows. Let $x=(x_1,\dots,x_n)\in \mathcal{O}(M)(n), $ and $y_i=(y_{i,1},\dots,y_{i,{k_i}})\in \mathcal{O}(M)(k_i)$ for $i=1,\dots, n$, we define 
    \[\alpha(x;y_1,\dots,y_n) = (x_1\cdot y_{1,1}, x_1\cdot y_{1,2}, \dots, x_1 \cdot y_{1,k_1},x_2\cdot y_{2,1},\dots, x_2,y_{2,k_2}, \dots, x_n \cdot y_{n,1},\dots, x_n \cdot y_{n,k_n}). \]
  More precisely, we take the tuples of products $x_i\cdot y_{i,j}$ ordered lexicographically, first by $i$ and then by $j$.
\end{definition}

 \begin{remark}
     Alternatively, for $1\le \ell \le k_+$, we can write the $\ell$-th component of $\alpha(x;y_1\dots,y_n)$ as
     \[
        \alpha(x;y_1\dots,y_n)_\ell = x_{\gamma(\ell)}\cdot y_{\gamma(\ell),\beta(\ell)},
     \]
     where $1\leq \gamma(\ell)\leq n$ is such that 
     \begin{equation}\label{equation: operad structure for monoids}
     \sum\limits_{i=1}^{\gamma(\ell)-1} j_i< \ell \leq \sum\limits_{i=1}^{\gamma(\ell)} j_i
     \end{equation}
     and $\beta(\ell) = \ell - \sum\limits_{i=1}^{\gamma(\ell)-1} j_i$.  Note that choices of $\beta$ and $\gamma$ are such that they line up with \cref{subsection:distributivity bijection}.
 \end{remark}

\begin{remark}\label{rem:not_free} The $\Sigma$-action is not free, as it fixes constant tuples. Thus, as a minimum, we need to remove the diagonal subgroup from $M^n$ in order to obtain a free action. However, removing the diagonal does not generally produce an operad, since the resulting spaces may not be closed under composition. Later, we rely on the notion of intersection monoid to produce related operads that are indeed $\Sigma$-free.
\end{remark}

We are interested in this construction in the presence of a $G$-action in order to produce $G$-operads. One could add a $G$-action to the monoid $M$, but this is not enough to ensure that the resulting operad $\cO(M)$ is a $G$-operad since the composition maps may not be equivariant, and the $G$-action may not fix the identity. To correct this issue, we will use the following definition as the correct equivariant generalization of monoids for our applications.

\begin{definition}
    A $G$-monoid is a monoid $M$ in $(\Set^G,\times,\ast)$.
\end{definition}

A monoid $M$ in $\Set^G$ is thus one in which the corresponding multiplication map $\mu\colon M\times M\to M$ is $G$-equivariant and where the unit $1\in M$ is $G$-fixed. This ensures that the composition maps $\gamma$ of $\cO(M)$ are also equivariant. We can then deduce that if $M$ is a $G$-monoid, then $\cO(M)$ is a well-defined $G$-operad.

\subsubsection{Pairings of monoids}
We now define pairings of monoids and show that they induce a pairing of operads.

\begin{definition}
\label{definition: pairing of monoids}
    A \emph{pairing of monoids} ($M, N, \xi$) is the data of two monoids $M$ and $N$ together with a map 
    \[
    	\xi \colon M\times N \to N
	\] 
	satisfying the following axioms:
    \begin{enumerate}
        \item\label{item-pairing-monoids-identity} (\emph{identity}) for any $m\in M$, and $n\in N$ we have $\xi(1_M,n)=n$, and $\xi(m,1_N)=1_N$,
        \item\label{item-pairing-monoids-associativity} (\emph{associativity}) for any $m_1,m_2\in M$ and $n\in N$ we have 
        	\[
				\xi(m_1,\xi(m_2,n))=\xi(m_1 \cdot m_2,n),
			\]
        \item\label{item-pairing-monoids-distributivity} (\emph{distributivity}) for any $m\in M$ and $n_1,n_2\in N$ we have 
        \[
            \xi(m,n_1 \cdot n_2) = \xi(m,n_1)\cdot \xi(m,n_2),
        \]

        \item\label{item-pairing-monoids-centrality} (\emph{centrality}) for any $m_1,m_2\in M$ and $n_1,n_2\in N$ we have     \[
                \xi(m_1,n_1) \cdot \xi(m_2,n_2) = \xi(m_2,n_2) \cdot \xi(m_1,n_1).
            \]
    \end{enumerate}
\end{definition}

\begin{remark}
\label{remark:compatible-pair-monoids-N-commutative}
The identity and centrality conditions force \( N \) to be a commutative monoid when \( m_1 = m_2 = 1_M \). 
As a result, if there is a pairing of $M$ and $N$, the monoid \( N \) must be commutative, and the action can be described as a monoid homomorphism \( \widetilde{\xi} \colon M \to \mathrm{End}(N) \), where \( \mathrm{End}(N) \) denotes the endomorphism monoid of \( N \).
\end{remark}

\begin{example}
Let us consider the monoids $A\coloneqq (\mathbb{N},+,0)$ and $M\coloneqq (\mathbb{N},\times ,1)$. Defining the structure map $\xi\colon M\times A\to A$ by $\xi(m,a)=a+\dots+a$, the addition of $a$ $m$-times. It is straightforward to check that this defines a pairing of the monoids $M$ and $A$. 
More generally, given any semiring $R$, commutativity of addition implies there is a pairing of the multiplicative monoid with the additive monoid.
\end{example}

We can use pairings of monoids to build pairings of operads in $(\Set,\times,\ast)$. Recall from \cref{notation:permutation in blocks} that the lexicographical ordering on tuples $\bm{q} = (q_1,\ldots,q_n) \in \prod_{i=1}^n \underline{j_i}$ gives us a preferred bijection
\[
	\underline{j_1} \times \cdots \times \underline{j_n} \cong \underline{j_\times},
\]
where $j_\times = \prod_{i=1}^n j_i$.
Using this, we have a preferred bijection  
\[
    \Map(\underline{j_\times}, M)
    \cong 
    \Map(\underline{j_1} \times \cdots \times \underline{j_n}, M).
\]
We will use this bijection implicitly from here. In particular, we view elements of $\mathcal{O}(M)(j_\times)$ as functions $\underline{j_1} \times \cdots \times \underline{j_n}\to M$. 
Explicitly, given $m$ be a non-negative integer, $M$ a monoid, and $x = (x_1,\dots,x_n) \in M^n$, we denote $x(i) = x_i$ for all $1 \leq i \leq n$. 

\begin{theorem}
\label{thm: operad pairing from compatible monoids} 
Let $(M,N,\xi)$ be a pairing of monoids. The family of maps
\[
	\lambda\colon \mathcal{O}(M)(k)\times \mathcal{O}(N)(j_1)\times \cdots \times \mathcal{O}(N)(j_k) \to \mathcal{O}(N)(j_\times), 
\]
determined on all $\bm{\ell}=(\ell_1,\dots,\ell_k)\in \underline{j_1} \times \cdots \times \underline{j_k}$ by 
\begin{align*}
    \lambda(p,x_1,\dots x_k)(\bm{\ell})= \prod_{r=1}^k\xi(p(r),x_r(\ell_r)),
\end{align*}
gives an operad pairing $(\cO(M),\cO(N),\lambda)$.
\end{theorem}
\begin{proof}
    Let us begin by setting some notation.  Suppose we are given a positive integer $k$, for all $1\leq r\leq k$ we fix positive integers $j_r$, and for all $1\leq q\leq j_r$ we fix a positive integer $i_{r,q}$.  Suppose we are given choices of elements
    \[
        p\in \mathcal{O}(M)(k),\quad p_r\in \cO(M)(j_r), \quad x_r\in \cO(N)(j_r), \quad x_{r,q}\in\cO(N)(i_{r,q})
    \]
    in our operads. Recall from \cref{notation:permutation in blocks} the shorthands 
    \[
        x_{J_r} = (x_{r,1},\dots, x_{r,j_{r}})\in \cO(N)(i_{r,1})\times\dots\times \cO(N)(i_{r,j_r})
    \]
    and when $\bm{q} = (q_1,\dots q_k)\in \underline{j_1}\times \dots\times \underline{j_k}$ also 
    \[
        x_{\bm{q}} = (x_{1,q_1},x_{2,q_2},\dots, x_{k,q_k})\in \cO(N)(i_{1,q_1})\times \dots\times \cO(N)(i_{k,q_k}).
    \]
    We assume that the collection of all such $\bm{q}$ is ordered according to the lexicographical ordering on $\underline{j_1}\times\dots\times \underline{j_k}$.

	We now show that the operation $\lambda$ satisfies the axioms of \cref{def:operad_pairing}.  Axioms \labelcref{action-id-P}-\labelcref{action-when-xi-ast}
    are straightforward and we omit the proofs, noting that centrality is needed for axiom \labelcref{action-sigma}. 
    
    We need to establish the following relations: 
    \begin{enumerate}
    \setlength{\itemsep}{8pt}
        \item\label{relation-item-compatible-pair-monoids-1} $\lambda(\mu(p;p_1,\dots,p_k);x_{J_1},\dotsm,x_{J_k}) = \lambda(p;\lambda(p_1;x_{J_1}),\dots,\lambda(p_k; x_{J_k}))$,
          \item\label{relation-item-compatible-pair-monoids-2} $\alpha(\lambda(p;x_1,\dots, x_k); \prod\limits_{\bm{q}} \lambda(p;x_{\bm{q}}))\circ \nu = \lambda(p; \alpha(x_1; x_{J_1}),\dots,\alpha(x_k;x_{J_k}))$,
    \end{enumerate}
    where $\nu$ is the distributivity permutation described in \cref{subsection:distributivity bijection}, and $\mu$ and $\alpha$ are the operad structure maps on $\cO(M)$ and $\cO(N)$, respectively.

   Let $n = \prod_{r=1}^k\prod_{q= 1}^{j_r}i_{r,q}$. To show \labelcref{relation-item-compatible-pair-monoids-1}, we must show that two elements in $\cO(N)(n)$ are the same. Via our preferred bijection $\underline{n}\cong \prod_{r=1}^k\prod_{q= 1}^{j_r}\underline{i_{r,q}}$, it suffices to fix an element $\bm{\ell}  = (\ell_{r,q})$ where $\ell_{r,q}\in \underline{i_{r,q}}$, and check that evaluating either side of the equation on $\bm{\ell}$ yields the same element of $N$.
    Evaluating the right hand side of \labelcref{relation-item-compatible-pair-monoids-1} 
    at $\bm{\ell}$ yields
    \begin{align*}
        \lambda(p;\lambda(p_1;x_{J_1}),\dots,\lambda(p_k; x_{J_k}))(\bm{\ell}) & = \prod\limits_{r=1}^k \xi(p(r),\lambda(p_r,x_{J_r})((\ell_{r,1},\dots,\ell_{r,j_r})))\\
        & = \prod\limits_{r=1}^k \xi\left(p(r),\prod\limits_{q=1}^{j_r}\xi(p_r(q),x_{r,q}(\ell_{r,q}))\right)\\
        & = \prod\limits_{r=1}^k \prod\limits_{q=1}^{j_r}\xi(p(r),\xi(p_r(q),x_{r,q}(\ell_{r,q})))\\
        & = \prod\limits_{r=1}^k \prod\limits_{q=1}^{j_r}\xi(p(r)\cdot p_r(q),x_{r,q}(\ell_{r,q})).
    \end{align*}
    Here, the first and second equality are the definition of $\lambda$, the third is 
    the distributivity of $\xi$,  
    and the fourth is the associativity of $\xi$, see \cref{definition: pairing of monoids}. 
    Evaluating the left hand side of \labelcref{relation-item-compatible-pair-monoids-1} 
    at $\bm{\ell}$ we obtain
    \[
         \prod\limits_{r=1}^k\prod\limits_{q= 1}^{j_r} \xi(\mu(p;p_1,\dots,p_k)(r,q), x_{r,q}(\ell_{r,q}))
    \]
    and the claim follows from the definition of $\mu$ in \cref{def:operad_from_M}, which gives us $\mu(p;p_1,\dots,p_k)(r,q) = p(r) \cdot p_r(q)$. 

    To establish \labelcref{relation-item-compatible-pair-monoids-2},  
    we employ the same strategy.  Let us write 
    \[
        m =  \prod_{r=1}^{k}\sum_{q=1}^{j_r} i_{r,q} =\sum_{\bm{q}} \prod_{r=1}^k i_{r,q_r}
    \]
    where the sum in the middle is over all $\bm{q} = (q_1,\dots,q_k)\in \underline{j_1}\times\dots\times \underline{j_r}$.  
    Both sides of \labelcref{relation-item-compatible-pair-monoids-2}  
    are elements in $\cO(N)(m)$, where we take the explicit presentation of $\underline{m}$ as 
    \[
        \underline{m}\cong \prod\limits_{r=1}^k 
        \underline{i_{r,+}},
    \]
    where $i_{r,+} = \sum_{q=1}^{j_r} i_{r,q}$. 
    The distributivity bijection
    \[
        \nu\colon  \prod\limits_{r=1}^k\underline{i_{r,+}}\xrightarrow{\cong} \coprod\limits_{\bm{q}}\ \prod\limits_{r=1}^k \underline{i_{r,q_r}},
    \]
    where the coproduct on the right is taken over all tuples $\bm{q} = (q_1,\dots,q_k)\in \underline{j_1}\times\dots\times \underline{j_r}$, 
    is the permutation described explicitly in \cref{subsection:distributivity bijection}.  For the remainder of this proof, we adopt the notation therein.  
    Fix $\bm{\ell} = (\ell_1,\ldots,\ell_k)$ where each $\ell_r\in \underline{i_{r,+}}$ and write $\bm{q}(\bm{\ell})$ for the unique $\bm{q}$ such that $\nu(\bm{\ell})$ is the $\bm{q}$-block of the target of $\nu$.  
    Evaluating the right hand side of \labelcref{relation-item-compatible-pair-monoids-2},  
    at $\bm{\ell}$ and get
    \begin{align*}
         \lambda(p; \alpha(x_1; x_{J_1}),\dots,\alpha(x_k;x_{J_k}))(\bm{\ell}) & = \prod\limits_{r=1}^k \xi(p(r), \alpha(x_r,x_{J_r})(\ell_r))\\
         & = \prod\limits_{r=1}^k \xi(p(r), x_r(\gamma_r(\bm{\ell}))\cdot x_{r,\gamma_r(\bm{\ell})}(\beta_r(\bm{\ell}))\\
         & = \prod\limits_{r=1}^k \xi(p(r), x_r(\gamma_r(\bm{\ell})))\cdot\xi(p(r), x_{r,\gamma_r(\bm{\ell})}(\beta_r(\bm{\ell})))\\
         & = \left(\prod\limits_{r=1}^k \xi(p(r), x_r(\gamma_r(\bm{\ell})))\right)\cdot\left( \prod\limits_{r=1}^k\xi(p(r), x_{r,\gamma_r(\bm{\ell})}(\beta_r(\bm{\ell})))\right)\\
         & = \lambda(p;x_1\dots,x_r)(\bm{q}(\bm{\ell}))\cdot \lambda(p;x_{\bm{q}(\bm{\ell})})(\beta_1(\bm{\ell}),\dots,\beta_{k}(\bm{\ell})) \\
         & = \alpha(\lambda(p;x_1\dots,x_r);\prod\limits_{\bm{q}} \lambda(p;x_{\bm{q}}))(\nu(\bm{\ell}))
    \end{align*}
    where the first equality is the definition of $\lambda$, the second is the definition of $\alpha$, 
    the third is the distributivity of $\xi$, 
    the fourth is the centrality of $\xi$, see \cref{definition: pairing of monoids} for these, 
    the fifth is the definition of $\lambda$, and the sixth uses the definitions of $\alpha$ and $\nu$.  Since the final term is exactly the left hand side of equation \labelcref{relation-item-compatible-pair-monoids-2},  
    the proof is complete.
\end{proof}

\begin{remark}\label{rem:Gaction_monoids}
    If $M$ and $N$ are $G$-monoids and the pairing $\xi\colon M\times N\to N$ is such that $\xi(g\cdot m,g\cdot n) = g\cdot\xi(m,n)$ for all $m\in M$, $n\in N$, and $g\in G$, then tracing through the proof above shows that $\lambda$ gives a pairing of operads in $\Set^G$.
\end{remark}

\subsection{Operads from intersection monoids}\label{subsec:operads-from-intersection-monoids} A natural question to ask is if the operad pairing $(\cL,\cK)$ of the linear isometries operad and the Steiner operad (see \cref{prop:operad_pairing_linearisometries_Steiner}) can be obtained as one of the above. This is not the case, as we show below. However, we can use \emph{intersection monoids} \cite{szczesny2025realizingtransfersystemssuboperads} to adapt the constructions above in a way that recovers this example. Furthermore,  when we have an intersection monoid, we obtain control over the structure of the operad in a manner similar to an embedding operad (such as the little disks). We later use this to build interesting suboperads of $\cO(M)$ that will aid in the goal of realizing operad pairings. 

\begin{remark}
There is no pairing of the monoids $(\cL(1),\cK(1))$, 
as the centrality axiom cannot hold. 
Indeed, note that the pointwise composition of two Steiner paths does not commute, so $\cK(1)$ is not commutative 
and therefore there can be no pairing (\cref{remark:compatible-pair-monoids-N-commutative}). As we will see later, one can indeed construct a pairing that satisfies all axioms except for centrality.
\end{remark} 

\begin{definition}\label{definition: intersection monoid}
    An \emph{intersection monoid} is a monoid $M$ with a symmetric relation $\vee$ such that the following holds. 
    \begin{enumerate}
        \item The complement relation\footnote{The complement relation $\wedge$ is defined as $x\wedge y$ if and only if $x$ is not related to $y$ by $\vee$.} $\wedge$ is reflexive.
        \item If $x_1,x_2,y_1,y_2\in M$ and $x_1\vee x_2$, then $(x_1 \cdot y_1) \vee (x_2 \cdot y_2)$.
        \item If $x_1,x_2,y\in M$ and $x_1\vee x_2$, then $(y \cdot x_1) \vee (y \cdot x_2)$.
    \end{enumerate}
    If $x\vee y$ we will say that $x$ and $y$ are \emph{disjoint}. If $x\wedge y$ then we will say that $x$ and $y$ \emph{intersect}. An intersection monoid has a \emph{non-trivial intersection relation} if there exists at least one pair of disjoint elements. 
    An \emph{intersection $G$-monoid} is a $G$-monoid $M$ whose underlying monoid structure in $\Set$ is an intersection monoid such that if $x_1,x_2\in M$ with $x_1\vee x_2$, then $(g \cdot x_1) \vee (g \cdot x_2)$ for all $g\in G$. 
\end{definition}

\begin{remark}
    For any intersection monoid $((M,\cdot, 1), \vee)$, the identity element $1$ intersects all elements of $M$. Indeed, if there exists $x\in M$ such that $x\vee 1$ then by axiom (b) in the definition we should have $x\cdot 1\vee 1\cdot x$, which contradicts axiom (a), the reflexivity of the complement relation. 
\end{remark}

\begin{remark}
    If $M$ is an intersection monoid with a non-trivial intersection relation, then one can show that for any positive integer $n$ there always exists a collection of elements $x_1, \dots, x_n$ that are pairwise disjoint. See \cite[Lemma 3.2]{szczesny2025realizingtransfersystemssuboperads} for a proof of this. 
\end{remark}

The definition and terminology are motivated by the following examples. 

\begin{example}
	Let $X$ be a set and let $M$ be the monoid of injective functions $f \colon X \hookrightarrow X$ under composition. This is an intersection monoid with $\vee$ given by $f \vee g$ if and only if $f$ and $g$ have disjoint images. The complement relation is $f \wedge g$ if and only if their images intersect. 

    If $X$ is a $G$-set, then we can give $M$ a $G$-action via conjugation. In this case, $M$ then becomes an example of an intersection $G$-monoid.
\end{example}

\begin{example}\label{ex:intersection_monoids}
    Some important examples come from the unary components of well-known operads. For example, the following are intersection monoids.
    \begin{enumerate}
        \item If $\mathcal{C}_k$ is the operad of little $k$-cubes, then $\mathcal{C}_k(1)$ is an intersection monoid, where for $x,y\in \mathcal{C}_k(1)$ we set $x\vee y$ if and only if 
        \[ 
            x(\mathrm{int}(I^k))\cap y(\mathrm{int}(I^k))=\emptyset.
        \]
        In fact, any embedding operad gives an intersection monoid in this way. 
        \item The previous example also extends to the Steiner operad $\cK$. Given two Steiner paths $h_1,h_2\in\cK(1)$, then we set $h_1\vee h_2$ if and only if 
        \[
            \im(h_1(0)) \cap \im(h_2(0)) = \emptyset. 
        \]
        \item For the linear isometries operad $\mathcal{L}$, the unary space $\mathcal{L}(1)$ is an intersection monoid, where for $x,y\in \mathcal{L}(1)$ we have $x\vee y$ if and only if $x(\RR^\infty)\perp y(\RR^\infty)$. 
    \end{enumerate}
\end{example}

We can use intersection monoids to build suboperads of the operads associated with monoids in~\cref{def:operad_from_M}. 

\begin{definition}\label{def:operad_from_IM}
    Let $M$ be an intersection monoid. We define a suboperad $\cO^\vee(M)$ of $\cO(M)$ by taking
    \[
    	\mathcal{O}^\vee(M)(n)
			\coloneqq
			\{
				(x_1,\dots,x_n) \in M^n 
					\mid 
				x_i\vee x_j\text{ for all }i\neq j
			\}.
	\]
Note that the composition of $\cO(M)$ preserves the pairwise disjointness of the elements, and so we indeed have a suboperad. When $M$ is an intersection $G$-monoid, then the spaces $\cO^\vee(M)(n)$ are preserved by the $G$-action, and thus $\cO^\vee(M)$ is a well-defined $G$-suboperad of $\cO^\vee(M)$.
\end{definition}

\begin{remark}
    It is straightforward to check that the $\Sigma$-action on $\cO(M)$ described in~\cref{def:operad_from_M} 
    restricts to a $\Sigma$-action on $\cO^{\vee}(M)$. 
    Moreover, this restricted action is now free, as promised in \cref{rem:not_free}. 
\end{remark}

\begin{example}
    The operads $\mathcal{O}^{\vee}(M)$ constructed from the intersection monoids in \cref{ex:intersection_monoids} are just (discrete versions of) the original operads.  That is, for every $n\geq 0$, the set $\mathcal{O}^\vee(\mathcal{K}(1))(n)$ is the underlying set of the Steiner operad space $\mathcal{K}(n)$, and similarly for the linear isometries operad. 
    In fact, $\mathcal{O}^{\vee}(\mathcal{K}(1))=\mathcal{K}$ and $\mathcal{O}^{\vee}(\mathcal{L}(1))=\mathcal{L}$. 
\end{example}

\subsubsection{Pairings of intersection monoids}
We conclude this section by refining the previous results in order to obtain pairings of operads from pairings of intersection monoids. With this method, we recover the well-known pairing of operads $(\cL,\cK)$, see \cref{cor:linearisometries_Steiner_pair_intersection_monoids}. 

\begin{definition}
\label{definition: pairing of intersection monoids}
    A \emph{pairing of intersection monoids} $(M,N)$ consists of two intersection monoids $M$ and $N$ together with a map $$\xi\colon M\times N\to N$$
    satisfying the following axioms:
 \begin{enumerate}
        \item (\emph{identity}) for any $m\in M$ and $n\in N$ we have $\xi(1_M,n)=n$ and $\xi(m,1_N)=1_N$, 
        \item (\emph{associativity}) for any $m_1,m_2\in M$ and $n\in N$ we have 
        	\[
				\xi(m_1,\xi(m_2,n))=\xi(m_1 \cdot m_2,n), 
			\]
        \item (\emph{distributivity}) for any $m\in M$ and $n_1,n_2\in N$ we have 
        \[
            \xi(m,n_1 \cdot n_2) = \xi(m,n_1)\cdot \xi(m,n_2), 
        \]
        \item (\emph{centrality}) for any  $m_1\vee m_2\in M$ and $n_1,n_2\in N$ we have
        	\[
				\xi(m_1,n_1) \cdot \xi(m_2,n_2)=\xi(m_2,n_2) \cdot \xi(m_1,n_1), 
			\]
        \item\label{item-pairing-int-monoids-disjunction} (\emph{disjunction}) if $n_1\vee n_2$ then for any $m\in M$ we have 
        	\[
				\xi(m,n_1)\vee \xi(m,n_2).
			\]
    \end{enumerate}
\end{definition}

\begin{remark}
Comparing \cref{definition: pairing of intersection monoids} with \cref{definition: pairing of monoids}, we see that conditions \labelcref{item-pairing-monoids-identity},  
\labelcref{item-pairing-monoids-associativity},
and \labelcref{item-pairing-monoids-distributivity} 
are exactly the same, whereas \labelcref{item-pairing-monoids-centrality} 
is weakened, and \labelcref{item-pairing-int-monoids-disjunction} 
is new. 
\end{remark}

\begin{remark}
    We write $\mathrm{End}_\vee(N)$ for the set of $\vee$-preserving endomorphisms of $N$.
    Then the above condition is equivalent to giving a monoid homomorphism $\tilde\xi: M \to \mathrm{End}_\vee(N)$ such that $\tilde\xi(m_1)(n_1)$ and $\tilde\xi(m_2)(n_2)$ commute in $N$ for all $n_1,n_2\in N$  whenever $m_1\vee m_2$. 
\end{remark}

In what follows, we present an example of this structure that will play a fundamental role for our purpose of realizing operad pairings. 

\begin{proposition}\label{lem:operad_pairing_linearisometries_Steiner}
There is a pairing of intersection monoids $(\cL(1),\cK(1),\xi)$, where the action map $\xi$ is the $1$-dimensional case of the pairing map $\lambda$ in the operad pairing $(\cL,\cK,\lambda)$ (see \cref{prop:operad_pairing_linearisometries_Steiner}). 
\end{proposition} 

\begin{proof}
Let us explicitly recall the map $\xi\colon \cL(1)\times\cK(1)\to \cK(1)$. Given a linear isometry $f\in \cL(1)$ and a Steiner path $h\in\cK(1)$, the action is given by $\xi(f,h)=f\ast h$. Recall that $f\ast h$ is the Steiner path, where $(f\ast h)(t)$ is the composite 
\begin{center}
\begin{tikzcd}
	\RR^\infty\arrow[r,"="]&\im(f)\oplus \im(f)^\perp\ar[d,"f^{-1}\oplus \id"] & \im(f)\oplus \im(f)^\perp \arrow[r,"="]&\RR^\infty.\\ 
	&(\RR^\infty)\oplus \im(f)^\perp\arrow[r,"h(t)\oplus \id"] & (\RR^\infty)\oplus \im(f)^\perp \arrow[u,"f\oplus \id"]&
\end{tikzcd}
\end{center}

The verification of the identity, associativity, and distributivity axioms is straightforward. 
For the centrality axiom, let $f_1\vee f_2\in \cL(1)$, let $h_1,h_2\in \cK(1)$, let $\alpha= \xi(f_1,h_1)(t)$, and let $\beta = \xi(f_2,h_2)(t)$. Observe from the construction of the maps $\alpha$ and $\beta$ that we have the following: 
\begin{align*}
	\alpha(\im(f_1)) \subseteq \im(f_1)\text{ and } \alpha\vert_{\im(f_1)^\perp} = \id_{\im(f_1)^\perp},  \\ 
	\beta(\im(f_2)) \subseteq \im(f_2)\text{ and } \beta\vert_{\im(f_2)^\perp} = \id_{\im(f_2)^\perp}. 
\end{align*}
While $\beta$ (resp. $\alpha$) is not linear, note that it respects the addition of elements in $\im(f_1)$ (resp $\im(f_2)$) with elements in the complement. 

Since $f_1\vee f_2$, we have $\im(f_1)\perp\im(f_2)$ by definition. Hence, for $v\in \RR^\infty$ any vector, we can write it as a decomposition $v=v_1+v_2+v_3$ where $v_1\in \im(f_1)$, $v_2\in \im(f_2)$, and $v_3\in \im(f_1)^\perp\cap \im(f_2)^\perp$. Using the properties above of $\alpha$ and $\beta$, as well as $\im(f_1) \subseteq \im(f_2)^\perp$ and $\im(f_2) \subseteq \im(f_1)^\perp$, we have 
\begin{align*}
	\left(\alpha\circ \beta\right)(v) &= \left(\alpha\circ \beta\right)(v_1+v_2+v_3) \\ 
	&= \alpha\left(v_1 + \beta(v_2) + v_3\right) \\ 
	&= \alpha(v_1) + \beta(v_2) + v_3 \\ 
	&= \beta(\alpha(v_1) +v_2 + v_3) \\ 
	&= (\beta\circ\alpha)(v_1+v_2 + v_3) \\ 
	&= (\beta\circ \alpha)(v).
\end{align*} 
Thus, the centrality condition follows.

Finally, for the disjunction axiom, let $f\in \cL(1)$, let $h_1\vee h_2\in \cK(1)$, and let $x$ be an arbitrary point in $\mathbb{R}^{\infty}$. Under the decomposition $\mathbb{R}^{\infty}\cong \im(f)\oplus \im(f)^{\perp}$ we identify $x$ with a pair $(y,z)$. The maps $\xi(f,h_1)(0)$ and $\xi(f,h_2)(0)$ send the element $(y,z)$ to $(fh_1(0)f^{-1}y,z)$ and $(fh_2(0)f^{-1}y,z)$, respectively. Since the first two coordinates never agree, as $h_1\vee h_2$, these two elements are disjoint.
\end{proof}

\begin{theorem}\label{thm: operad pairing from compatible intersection monoids} Let $(M,N,\xi)$ be a pairing of intersection monoids. The family of maps
\[
	\lambda\colon \mathcal{O}(M)(k)\times \mathcal{O}(N)(j_1)\times \cdots \times \mathcal{O}(N)(j_k) \to \mathcal{O}(N)(j_{\times}), 
\]
of  \cref{thm: operad pairing from compatible monoids} restricts to give an operad pairing on the pair of suboperads $(\cO^\vee(M), \cO^\vee(N),\lambda)$. 
\end{theorem}

\begin{proof} 
Let us begin by setting some notation. Suppose that we are given a positive integer $k$, for all $1\leq r\leq k$ we fix positive integers $j_r$, and for all $1\leq q\leq j_r$ we fix a positive integer $i_{r,q}$. Suppose that we are given choices of elements
    \[
        p\in \mathcal{O}(M)(k),\quad p_r\in \cO(M)(j_r), \quad x_r\in \cO(N)(j_r), \quad x_{r,q}\in\cO(N)(i_{r,q})
    \]
    in our operads. 
    Recall that the family of maps 
\[
	\lambda\colon \mathcal{O}(M)(k)\times \mathcal{O}(N)(j_1)\times \cdots \times \mathcal{O}(N)(j_k) \to \mathcal{O}(N)(j_{\times})
\]
is defined on all $\bm{\ell}=(\ell_1,\ell_2,\dots,\ell_k)\in \underline{j_1} \times \cdots \times \underline{j_k}$ by
\[
    \lambda(p,x_1, \dots x_k)(\bm{\ell})= \prod_{r=1}^k\xi(p(r),x_r(\ell_r)).
\]

	Let us first justify that the map $\lambda$ restricts to the suboperads of \cref{def:operad_from_IM}.
	That is, when the input is in $\mathcal{O}^\vee(M)(k)\times \mathcal{O}^\vee(N)(j_1)\times \cdots \times \mathcal{O}^\vee(N)(j_k) $, then  \[\lambda(p,x_1,\dots,x_k)\in \cO^\vee(j_\times).\] 
	This amounts to showing that given another $\bm{\ell}^\prime\in \underline{j_\times}$
    with $\bm{\ell}\neq\bm{\ell}^\prime$ then
	\begin{align}\label{eq:needed_condition_for_welldefined}
	\lambda(p,x_1,x_2,\dots,x_k)(\bm{\ell})\vee \lambda(p,x_1,x_2,\dots,x_k)(\bm{\ell}^\prime).	
	\end{align}
	Since $\bm{\ell}\neq\bm{\ell}^\prime$, there exists $s\in\underline{k}$ such that $\ell_s\neq\ell_s^\prime$, and so $x_s(\ell_s)\vee x_s(\ell_s^\prime)$ as $x_s\in \cO(j_s)$. 
	From the disjunction condition, we then have 
    \[
        \xi(p(s),x_s(\ell_s))\vee\xi(p(s),x_s(\ell^\prime_s)).
    \]
	Using the centrality condition and the fact that disjointness is preserved under right multiplication, it follows that 
    \[
        \prod_{r=1}^k\xi(p(r),x_r(\ell_r))\vee \prod_{r=1}^k\xi(p(r),x_r(\ell_r^\prime)).
    \]
    Namely, we have shown \cref{eq:needed_condition_for_welldefined}. 
    Hence, the map $\lambda$ is well-defined. 
    
    To show that the structure maps provide a pairing of operads, we follow the proof of \cref{thm: operad pairing from compatible monoids}. Note that although we don't have a global centrality condition, what we have suffices to achieve the chain of equalities at the end of said proof.
\end{proof}

\begin{corollary}\label{cor:linearisometries_Steiner_pair_intersection_monoids} There is a pairing of operads between $\cL$ and $\cK$. 
\end{corollary}

\begin{proof} This follows directly from the combination of \cref{thm: operad pairing from compatible intersection monoids} and \cref{lem:operad_pairing_linearisometries_Steiner}.
\end{proof}

When $M$ and $N$ are $G$-monoids such that $(M,N,\xi)$ is a pairing of intersection monoids, then in view of \cref{rem:Gaction_monoids}, it is sufficient to ask that the action $\xi$ is compatible with the $G$-actions to obtain a pairing of operads in $\Set^G$. In particular, this is the case when $M$ is the $1$-space of the equivariant linear isometries operad and $N$ is the $1$-space of the equivariant Steiner operad. 

\begin{proposition}\label{proposition: equivariant LI steiner pair}
    For any finite group $G$ and any $G$-universe $U$, there is a pairing of operads in $\Set^G$ given by the equivariant linear isometry and Steiner operads $(\cL_G(U),\cK_G(U))$. 
\end{proposition}

\section{Pairs of \texorpdfstring{$N_\infty$}{N-infinity}-operads from compatible indexing systems}\label{sec:Ninfinity_from_indexing}

In this section we seek a partial converse to \cref{prop:operad-action-induces-compatible-pair}, relying on the techniques from the previous section. In \cref{subsec:constructing_Ninfinity_from_Setoperads} we construct $N_\infty$-operads from $G$-operads in $\Set$, which we use to prove the desired result in the case of a complete additive component in \cref{theorem:mainthm}. We finish by constructing new examples from known ones in \cref{subsec:realizable_pairs_from_known}.

\subsection{Constructing \texorpdfstring{$N_{\infty}$}{N-infinity}-operads from operads in \texorpdfstring{$\Set$}{Set}}\label{subsec:constructing_Ninfinity_from_Setoperads}

We recall a way of producing $N_{\infty}$-operads from certain operads in the category $(\Set,\times,\ast)$.  This idea, originally due to Rubin \cite{RubinCombinatorial}, amounts to the observation that every set produces a contractible space in a canonical way.  Moreover, this construction preserves products and is compatible with group actions, and hence sends operads in $\Set$ to operads in $\Top^G$ that under certain circumstances are $N_\infty$.  We now make these ideas precise.

\subsubsection{Coinduced operads}\label{section:coinduced_operads}

Given a subgroup $H\leq G$, it is often the case that one starts with an operad in the category $(\Set^H, \times, \ast)$ and uses it to produce a $G$-equivariant operad for a finite group $G$. For this, we recall the coinduction-forgetful adjunction
\begin{equation}\label{equation:coinduction_of_sets}
    \begin{tikzcd}[column sep = huge]
	{\Set^G} & {\Set^H,}
	\arrow[""{name=0, anchor=east, inner sep=0}, "{i^\ast_H}", bend left=30, from=1-1, to=1-2]
	\arrow[""{name=1, anchor=center, inner sep=0}, "{\Coind_H^G}", bend left=30, from=1-2, to=1-1]
	\arrow["\dashv"{anchor=center, rotate=-90}, draw=none, from=0, to=1]
\end{tikzcd}
\end{equation} 
where the \emph{coinduction functor} $\Coind_H^G$ is given by the assignment $X\mapsto \Set^H(G,X)$ with the action defined as $(g \cdot f)(h) \coloneqq f(hg)$. Note that coinduction preserves products, since it is a right adjoint, and therefore it sends operads in $(\Set^H, \times, \ast)$ to operads in $(\Set^G, \times, \ast)$. 

\begin{notation}\label{not:coinduced_operad}
	Given an operad $\cO$ in $(\Set^H, \times, \ast)$, we denote the coinduced operad in $(\Set^G, \times, \ast)$ by $\Coind_H^G\cO$.  Explicitly, the $n$-th space of this $G$-operad is $(\Coind_H^G\cO)(n) = \Coind_H^G(\cO(n)$). When $H=\set{e}$ we will write $\cO_G:=\Coind_{\set{e}}^G\cO$. 
\end{notation}

\subsubsection{Towards $N_\infty$-operads}

We have shown that with $(\Set,\times,\ast)$ as a starting point, we can construct operads in $\Set^G$. We now want to discuss how to naturally move towards $\Cat$, and how these constructions interact. 

\begin{definition}
	Let $S$ be a set. The \emph{chaotic category} generated by $S$ is the category $\widetilde{S}$ with object set $S$ and a unique morphism between any two objects.
\end{definition}

\begin{remark}\label{rem:Bchaotic_contractible}
	The canonical inclusion of any single object of $S$ into $\widetilde{S}$ is a categorical equivalence, and hence the classifying space of $\widetilde{S}$ is always contractible.
\end{remark}

We have an adjunction 
\begin{equation}\label{Adj-Cat-Set}
\begin{tikzcd}[column sep = huge]
	{\Cat} & {\Set,}
	\arrow[""{name=0, anchor=east, inner sep=0}, "{\mathrm{ob}}", bend left=30, from=1-1, to=1-2]
	\arrow[""{name=1, anchor=center, inner sep=0}, "{(\widetilde{-})}", bend left=30, from=1-2, to=1-1]
	\arrow["\dashv"{anchor=center, rotate=-90}, draw=none, from=0, to=1]
\end{tikzcd}
\end{equation}
where the forgetful functor $\textrm{ob}\colon \Cat \to \Set$ sends a small category to its set of objects, and $(\widetilde{-})\colon \Set\to\Cat$ is the chaotic category functor. 

We denote by $BG$ the delooping of $G$, that is, the category with one object whose set of automorophisms is $G$. Since the functor $(-)^G\coloneqq \mathrm{Fun}(BG, -)$ underlies a $2$-functor, we know it preserves adjunctions and thus we obtain an adjunction

\begin{equation}\label{Adj-GCat-GSet}
\begin{tikzcd}[column sep = huge]
	{\Cat^G} & {\Set^G.}
	\arrow[""{name=0, anchor=east, inner sep=0}, "{\mathrm{ob}}", bend left=30, from=1-1, to=1-2]
	\arrow[""{name=1, anchor=center, inner sep=0}, "{(\widetilde{-})}", bend left=30, from=1-2, to=1-1]
	\arrow["\dashv"{anchor=center, rotate=-90}, draw=none, from=0, to=1]
\end{tikzcd}
\end{equation}

Furthermore, we observe that the adjunction~\eqref{Adj-Cat-Set} lifts to 
the adjunction~\eqref{Adj-GCat-GSet} through the forgetful functors as below. 

\begin{center}

\begin{tikzcd}[column sep = huge]
	{\Cat} &  {\Set}
	\arrow[""{name=0, anchor=east, inner sep=0}, "{\mathrm{ob}}", bend left=30, from=1-1, to=1-2]
	\arrow[""{name=1, anchor=center, inner sep=0}, "{(\widetilde{-})}", bend left=30, from=1-2, to=1-1]
	\arrow["\dashv"{anchor=center, rotate=-90}, draw=none, from=0, to=1]\\
     &\  \\
    {\Cat^G}\arrow[uu, "U"]\arrow[""{name=2, anchor=east, inner sep=0}, "{\mathrm{ob}}", bend left=30, from=3-1, to=3-2]& {\Set^G}\arrow[uu, "U"', inner sep=0]\arrow[""{name=3, anchor=center, inner sep=0}, "{(\widetilde{-})}", bend left=30, from=3-2, to=3-1]
	\arrow["\dashv"{anchor=center, rotate=-90}, draw=none, from=2, to=3] 
\end{tikzcd}
\end{center}

That is, both the chaotic category construction and taking underlying objects commute with the forgetful functors. Explicitly, if $S$ is a set with a $G$-action, then $\widetilde{S}$ becomes a category with a $G$-action, where the action on objects is induced by the action on $S$. 

\begin{remark}\label{remark: fixed points of classifying space}
	Let $S$ be a set with a $G$-action and let $B\widetilde{S}$ be the classifying space of the chaotic category $\widetilde{S}$.  
    When the group $G$ is finite, we obtain homeomorphisms
	\[
	(B\widetilde{S})^H\cong B(\widetilde{S}^H)\cong B(\widetilde{S^H}). 
	\]
	The first isomorphism comes from the fact that taking fixed points with respect to any subgroup is a finite limit, and that taking classifying space commutes with finite limits. The second isomorphism comes from the fact that forming chaotic categories is a right adjoint. In particular, note that the fixed points of $B\widetilde{S}$ are always either empty or contractible, depending on whether $S^H$ is empty or not empty. 
\end{remark}

\begin{proposition}[{cf.~\cite[Proposition 3.5]{RubinCombinatorial}}]\label{proposition: chaotic Eoo operads}
	Let $\cO$ be an operad in $(\Set^G,\times,\ast)$. Then
    \begin{enumerate}
        \item\label{item: chaotic O operads} the collection $\widetilde{\cO} \coloneqq \{\widetilde{\cO(n)}\}$ is an operad in $(\Cat^G,\times,\ast)$, and
        \item\label{item: chaotic EO operads} the collection $E\cO\coloneqq\{B\widetilde{\cO(n)}\}$ is an operad in $(\Top^G,\times,\ast)$.
        \end{enumerate}
        Moreover, if the action of $\Sigma_n$ on $\cO(n)$ is free for all $n\geq 0$, then $E\cO$ is an $N_{\infty}$-operad. If $(\cO_1,\cO_2)$ is a pairing of operads in $\Set^G$ then $(E{\cO}_1,E{\cO}_2)$ is a pairing of operads in $\Top^G$.
\end{proposition}

\begin{proof}
For \labelcref{item: chaotic O operads}, it is enough to recall that taking the chaotic category is a right adjoint, hence preserves products. Now \labelcref{item: chaotic EO operads} follows from \labelcref{item: chaotic O operads}, together with the fact that taking classifying spaces of categories preserves products as well.   
By \cref{remark: fixed points of classifying space}, if $\cO$ is $\Sigma$-free then $E\cO$ is also $\Sigma$-free. 
Applying \cref{remark: fixed points of classifying space} once again shows that all fixed point spaces are either empty or contractible, hence $E\cO$ is an $N_{\infty}$-operad when $\cO$ is $\Sigma$-free. 
To conclude the proof, we observe that the data of a pairing of operads is transported by strong monoidal functors, and that the composite functor $B\widetilde{(-)}$ is strong monoidal. 
\end{proof}

\begin{remark} 
By \cref{proposition: chaotic Eoo operads}, given a $\Sigma$-free operad $\cO$ in $(\Set,\times,\ast)$, its coinduced operad $\cO_G$ determines an $N_\infty$-operad $E\cO_G$ whose indexing system is the complete one. Given an intersection monoid $M$, since the operad $\cO^{\vee}(M)$ is $\Sigma$-free, we obtain an $N_{\infty}$-operad $E\cO^{\vee}_G(M)$ in $(\Top^G,\times, \ast)$ whose indexing system is the complete one. 
\end{remark}

We use this idea to build pairings of $N_\infty$-operads with prescribed indexing systems from pairings of intersection monoids.  

\begin{theorem}[{\cite[Theorem A]{szczesny2025realizingtransfersystemssuboperads}}]
\label{theorem: realization from monoids}
    Let $M$ be an intersection monoid with non-trivial intersection relation 
    and let $\cI$ be a $G$-indexing system. 
    The operad $E\cO^{\vee}_G(M)$ has a $N_\infty$-suboperad $E\cO^{\vee}_G(M)_{\cI}$ 
    whose associated indexing system is $\cI$. 
\end{theorem}

\begin{corollary}
\label{theorem:mainthm}
    Let $\cI$ be an indexing system of $G$ and let $(M,N)$ be a pairing of intersection monoids. 
    Then we have a pairing of operads $(E\cO^{\vee}_G(M)_{\cI},E\cO^{\vee}_G(N))$ whose associated compatible pair of indexing systems is $(\cI,\cI_c)$, where $\cI_c$ is the complete indexing system of \cref{example: example of indexing systems 1}~\labelcref{indexing system 1}.
\end{corollary}

\begin{proof}
	If $(M,N)$ is a pairing of intersection monoids, then by \cref{thm: operad pairing from compatible intersection monoids} we have the operad pairing $(\mathcal{O}^{\vee}(M),\mathcal{O}^{\vee}(N))$. Since coinduction is a right adjoint, we immediately get that $(E\mathcal{O}^{\vee}_G(M),E\mathcal{O}^{\vee}_G(N))$ is a pairing of $G$-operads. The result follows by restricting the pairing to the suboperad $E\cO_G(M)_\cI$ of \cref{theorem: realization from monoids}. 
\end{proof}

\subsection{Constructing new pairs from known cases}\label{subsec:realizable_pairs_from_known}
We now conclude the paper by presenting several ways of constructing new pairings of operads from known pairings of operads. 

\begin{definition}
We say that a compatible pair of indexing systems $(\mathcal{I}_1,\mathcal{I}_2)$ is \emph{realizable} if there exists a pairing of $N_{\infty}$-operads $(\cO_1,\cO_2)$ such that the indexing systems associated to $\cO_1$ and $\cO_2$ are $\cI_1$ and $\cI_2$, respectively. 
\end{definition}

 For instance, \cref{theorem:mainthm} implies that every compatible pair of indexing systems $(\cI_1,\cI_2)$ where $\cI_2$ is complete is realizable.

\begin{example}\label{example of realizable pairs 1}
    For any $G$-universe $U$, the linear isometries and Steiner operads are compatible by \cref{proposition: equivariant LI steiner pair}, and thus determine realizable indexing systems $(\cI_{\cL_G(U)},\cI_{\cK_G(U)})$.  Examples of indexing systems realized by these operads for some cyclic groups, $\Sigma_3$, and $Q_8$, are given in \cite{RubinDetecting}, although they are given in the language of transfer systems.
\end{example}

\begin{example}\label{example of realizable pairs 2}
     Let $f\colon G\to H$ be any group homomorphism and let $\cO$ be any $H$-$N_{\infty}$-operad.  Restricting along $f$ yields a $G$-$N_{\infty}$-operad $f^{*}\cO$.  Since restriction preserves products, any pairing of $H$-$N_{\infty}$-operads $(\cO_1,\cO_2)$ yields a pairing of $G$-$N_{\infty}$-operads $(f^*\cO_1,f^*\cO_2)$. 
\end{example}

In the remainder of this section we give some tools for approaching \cref{conjecture: realizability}. 
Specifically, we describe techniques which allow us to show that certain classes of compatible pairs of indexing systems are realizable. While we are not able to prove the conjecture, we produce several new families of examples. 

The first tool we use is the Cartesian product of operads. 
\begin{lemma}\label{lemma:products_of_pairs_are_pairs}
    Let $(\cP_1,\cQ_1,\lambda_1)$ and $(\cP_2,\cQ_2,\lambda_2)$ be pairings of $G$-operads. Then there is a pairing of $G$-operads between $\cP_1\times \cP_2$ and $\cQ_1\times \cQ_2$, with structure map induced by $\lambda_1 \times \lambda_2$. 
\end{lemma}

\begin{proof}
    The pairing maps are given by
    \begin{align*}
        (\cP_1\times \cP_2)(k)\times \prod\limits_{i=1}^k(\cQ_1\times \cQ_2)(j_i) & \cong \left( \cP_1(k)\times \prod\limits_{i=1}^k\cQ_1(j_i)\right)\times \left(\cP_2(k)\times \prod\limits_{i=1}^k \cQ_2(j_i)\right)\\
        & \xrightarrow{\lambda_1\times \lambda_2} \cQ_1(j_\times)\times \cQ_2(j_\times)\\
        & = (\cQ_1\times \cQ_2)(j_{\times})
    \end{align*}
    To check the required compatibility, note that it suffices to check it in the $\cQ_1$ and $\cQ_2$ components separately.  But this amounts to observing that $\lambda_1$ and $\lambda_2$ are pairings, which is true by assumption.
\end{proof}

Recall that the collection of indexing systems of a finite group $G$ forms a poset under inclusion.  In fact, this poset is a finite lattice with meets ($\wedge$) and joins ($\vee$). The meet structure can be understood at the level of operads.  Precisely, through the equivalence $\mathrm{Ho}(N_{\infty})\cong \mathrm{Index}(G)$ of \cref{theorem: indexing systems are operads}, the meet in the lattice structure of indexing systems corresponds to the product of $N_{\infty}$-operads, see \cite[Proposition 5.1]{BlumbergHillOperadic}. 

\begin{corollary}\label{corollary: realizability is closed under meets}
    If $(\cI_1,\cI_2)$ and $(\cI_1',\cI_2')$ are realizable pairs of indexing systems then $(\cI_1\wedge \cI_1',\cI_2\wedge \cI_2')$ is also realizable.
\end{corollary}

\begin{proof}
    Immediate from the discussion in the preceeding paragraph and \cref{lemma:products_of_pairs_are_pairs}. 
\end{proof}

If $(\cI_1,\cI_2)$ and $(\cI_1',\cI_2)$ are compatible pairs of indexing systems, then Blumberg--Hill showed that $(\cI_1\vee \cI_1',\cI_2)$ is also a compatible pair, see \cite[Proposition 7.84]{BlumbergHillBiincomplete}.  In particular, this implies that for any indexing system $\cI$ there is a unique largest indexing system $\mathrm{Hull}(\cI)$ such that $(\mathrm{Hull}(\cI),\cI)$ is compatible.  This indexing system is called the \emph{multiplicative hull} of $\cI$.

\begin{proposition}
    If $(\mathrm{Hull}(\cI),\cI)$ is realizable and $(\cI',\cI)$ 
    is a compatible pair of indexing systems, 
    then $(\cI',\cI)$ is realizable. 
\end{proposition}
\begin{proof}
    Let $\cI_c$ denote the complete indexing system.  By \cref{theorem:mainthm}, the pair $(\cI',\cI_c)$ is realizable.  Applying \cref{corollary: realizability is closed under meets} to the realizable pairs $(\mathrm{Hull}(\cI),\cI)$ and $(\cI',\cI_c)$ yields the claim.
\end{proof}

\begin{corollary}
    \label{conjecture reduced}
    The statement of \cref{conjecture: realizability} is equivalent to the claim that $(\mathrm{Hull}(\cI),\cI)$ is realizable for any $G$-indexing system $\cI$.
\end{corollary}

The last construction to build new pairings from old ones is with coinduction from \cref{section:coinduced_operads}. 

\begin{lemma}\label{lemma coinduced operad pairs}
    Let $H\leq G$ be a subgroup and let  $(\cO_1,\cO_2)$ be a pairing of $H$-$N_\infty$-operads. Then the coinduced operads \((\Coind_H^G\cO_1,\Coind_H^G\cO_2) \) are a pairing of $G$-$N_\infty$-operads.
\end{lemma}
\begin{proof}
    The forgetful-coinduction adjunction between sets and $G$-sets in \labelcref{equation:coinduction_of_sets} extends to an adjunction between spaces and $G$-spaces, and this, in turn, gives us the coinduction functor $\Coind_H^G$ from operads in $\Top^H$ to $\Top^G$.  Blumberg-Hill prove  that this coinduction functor will further restrict to $N_\infty$-operads \cite[Proposition 6.15]{BlumbergHillOperadic}. Since the coinduction functor on spaces is a right adjoint, it preserves products, and therefore pairings of operads.
\end{proof}

Given an $H$-$N_\infty$-operad $\cO$, let us write $\to_\cO$ for the corresponding transfer system. We wish to be able to determine the transfer system $\to_{\Coind_H^G\cO}$ from $\to_\cO$.
Before we do this, let us recall some elementary properties related to $G$-sets.
\begin{property}\label{prop:fixed_points_map}
    For a subgroup $H\leq G$, and $G$-set $Y$, there is a bijection \[\Set^G(G/H,Y) \cong Y^H\] given by $f\mapsto f(eH)$. 
\end{property}
For an $(H\times\Sigma_n)$-set $X$, the coinduced $G$-set $\Set^{H}(G,X)$ is a $G\times \Sigma_n$-set where the action is given by \[((g,\sigma)\cdot f)(g^\prime)=\sigma\cdot f(g^\prime g).\] 
\begin{property}\label{prop:coinduced_equiv}
    There is an isomorphism of $G\times \Sigma_n$-sets \[
    \Set^{H\times\Sigma_n}(G\times\Sigma_n,X) \cong \Set^{H}(G,X)
    \] given by $f\mapsto i^\ast_Gf$ where $i_G\colon G\to G\times \Sigma_n$ is the inclusion $i_G(g)=(g,1)$. 
\end{property}

In what follows, we will use the conjugate subgroup notation $H^{g}:=g^{-1}Hg$. The following property is the well-known Mackey double coset formula. Recall that $i_H^\ast$ is the restriction functor from $G$-sets to $H$-sets.
\begin{property}\label{prop:mackey_formula}
    Suppose we have subgroups $H,K\leq G$. There is an isomorphism of $H$-sets \[i^\ast_H\left(G/K\right)\cong \coprod_{g\in [K\backslash G/H]}H/(H\cap K^g).\] Here $g\in [K\backslash G/H]$ signifies we are indexing over some choice of double coset representatives.
\end{property}

We can use these properties to get a useful description of graph subgroup fixed points associated with the coinduced sets $\Set^{H}(G, X)$ when $X$ is a $(H\times\Sigma_n)$-set.

\begin{lemma}[\protect{cf. \cite[Proposition 6.16]{BlumbergHillOperadic}}]\label{lem:coinduced_fixed_points}
    Let $H\leq G$, and $X$ be a left $H\times \Sigma_n$-set. For a graph subgroup $\Gamma \leq G\times \Sigma_n$ corresponding to a homomorphism $K\to \Sigma_n$, there is a bijection \[
    \left(\Set^{H}(G, X)\right)^\Gamma \cong \prod_{g\in [K\backslash G /H]}X^{(H\times \Sigma_n) \cap \Gamma^{(g,1)} } .
    \]
\end{lemma}
\begin{proof}
Consider the following sequence of isomorphisms$\colon$
    \begin{align*}
        \left(\Set^{H}(G, X)\right)^\Gamma &\cong \left(\Set^{H\times \Sigma_n}(G\times \Sigma_n, X)\right)^\Gamma \\ 
        &\cong \Set^{G\times\Sigma_n}\left((G\times \Sigma_n)/\Gamma, \Set^{H\times \Sigma_n}(G\times \Sigma_n, X)\right) \\ 
        &\cong \Set^{H\times \Sigma_n}(i^*_{H\times \Sigma_n}((G\times \Sigma_n)/\Gamma), X).
    \end{align*}
    Here, the first two isomorphisms are \cref{prop:fixed_points_map,prop:coinduced_equiv}, and the third isomorphism is the restriction-coinduction adjunction.

    We decompose the left $(H\times \Sigma_n)$-set $i^*_{H\times \Sigma_n}(G\times \Sigma_n/\Gamma)$ using the Mackey double coset formula (\cref{prop:mackey_formula}). Note that we may pick double coset representatives for $\Gamma\backslash (G\times \Sigma_n)/(H\times\Sigma_n)$ that have the form $(g,1)$, since we can always multiply on the right by elements of the form $(1,\sigma)$. Thus, we have isomorphisms:

    \begin{align*}
        \left(\Set^{H}(G, X)\right)^\Gamma &\cong \Set^{H\times \Sigma_n}(i^*_{H\times \Sigma_n}(G\times \Sigma_n)/\Gamma, X) \\
        &\cong \prod_{(g,1)\in [\Gamma \backslash (G\times\Sigma_n) / (H\times\Sigma_n)]}\Set^{H\times \Sigma_n}((H\times\Sigma_n) /((H\times \Sigma_n)\cap \Gamma^{(g,1)}),X) \\ &\cong \prod_{(g,1)\in [\Gamma \backslash (G\times\Sigma_n) / (H\times\Sigma_n)]}X^{(H\times \Sigma_n) \cap \Gamma^{(g,1)}}.
    \end{align*}

    The lemma then follows by observing that if $\set{(g_k,1)}$ is a complete system of double coset representatives of $\Gamma \backslash (G\times\Sigma_n) / (H\times\Sigma_n)$, then $\set{g_k}$ is a complete system of double coset representatives of $K\backslash G /H$.
\end{proof}

\begin{definition}\label{def coinduced transfer system}
    Given a subgroup $H\leq G$ and a $H$-transfer system $\to$, we define the \emph{$G$-coinduced transfer system} $\to^G$ as follows. We write $L \to^G K$ whenever $L \leq K \leq G$ and for all $g \in G$ there exists a transfer
\[
    H \cap L^{g} \;\to\; H \cap K^{g}.
\]
\end{definition}

Both the name of $\to^G$, and the fact that it is a well-defined $G$-transfer system come from the following proposition.

\begin{proposition}
 Let $H\leq G$, and suppose that we have an $H$-$N_\infty$-operad $\cO$ with associated transfer system $\to_\cO^G$. Then the relations $\to_{\Coind_H^G\cO}$ and $\to_\cO^G$ are the same.
\end{proposition}

\begin{proof}
     There is a transfer $L\to_{\Coind_H^G\cO} K$ if and only if $\Coind_H^G\cO(n)^{\Gamma}$ is non-empty, where $\Gamma\subset G\times \Sigma_n$ is the graph subgroup corresponding to some choice and order of coset representatives for the $K$-set $K/L$ (where $n=\vert K/L\vert$). From \cref{lem:coinduced_fixed_points}, we deduce that 
     \[\Coind_H^G\cO(n)^{\Gamma}=\left(\Set^H(G,\cO(n))\right)^\Gamma\]
     is non-empty if and only if for all $g\in G$, the set $\cO(n)^{(H\times \Sigma_n)\cap\Gamma^{(g,1)}}$ is non-empty. Notice that $(H\times \Sigma_n)\cap\Gamma^{(g,1)}$ is itself a graph subgroup corresponding to the $(H\cap K^g)$-set $i^\ast_{(H\cap K^g)}(K^g/L^g)$. As a consequence, $\cO(n)^{(H\times \Sigma_n)\cap\Gamma^{(g,1)}}$ is non-empty if and only if 
     \begin{equation}\label{equation:restricted_big_condition}
         i^\ast_{(H\cap K^g)}(K^g/L^g)\in\cI_\cO(H\cap K^g).
     \end{equation}
     The double coset formula (\cref{prop:mackey_formula}) gives us a decomposition of $(H\cap K^g)$-sets 
     \[
         i^\ast_{(H\cap K^g)}(K^g/L^g) \cong \coprod_{x\in[L^g\backslash K^g/(H\cap K^g)]}(H\cap K^g)/(H\cap L^{gx}).
     \] 
     Since indexing systems are closed under disjoint unions and subobjects (\cref{definition:indexing_systems}), we conclude that \labelcref{equation:restricted_big_condition} holds if and only if for all $x\in K^g$ that \[
     (H\cap K^g)/(H\cap L^{gx})\in \cI_\cO(H\cap K^g),
     \]
     which is equivalent to $H\cap L^{gx}\to_\cO H\cap K^g$. Tracing our chain of equivalences, we have shown that $L\to_{\Coind_H^G\cO} K$ if and only for all $g\in G$ and $x\in K^g$, that $H\cap L^{gx}\to_\cO H\cap K^g$. The lemma then follows as $K^{gx}=(K^g)^x=K^g$ for all $x\in K^g$.

\end{proof}

\begin{example}
    Recall from \cref{example:basic_transfer_systems}, that given any non-trivial subgroup $J\leq G$, there are two $J$-transfer systems: the complete $J$-transfer system $\to_c^J$ and the trivial $J$-transfer system $\to_0^J$. The coinduced $G$-transfer system $(\to_c^J)^G$ is just the complete $G$-transfer system $\to^G_c$, and we don't get anything new in this case. 
    In contrast, the coinduced $G$-transfer system $(\to_0^J)^G=: \;\to_J$ from the trivial $J$-transfer system is usually non-trivial. 
    Moreover, it has a particularly nice description. For any $K\leq H\leq G$ we have $K\to_J H$ if and only if for all $g\in G$ it holds that $J^g\cap H\leq K$. To see this, we have $K\to_J H$ if and only if for all $g\in G$ we have $K^{g}\cap J \to_0^J H^g\cap J$. Since $\to_0^J$ is the trivial $J$-transfer system, we know that $K^{g}\cap J \to_0^J H^g\cap J$ if and only if the equality $K^{g}\cap J = H^g\cap J$ holds, which is true if and only if $H^g\cap J \leq K^{g}$ since $K^{g}\leq H^g$.  We call $\to_J$ the \emph{$J$-local transfer system}.
\end{example}

    We end with an example which describes the realizability of compatible pairs of transfer systems for the group $G = \Sigma_3$. Recall from \cref{def compatible pair transfer system} the definition of a compatible pair of transfer systems.

    \begin{example}
        \cref{table: S3 transfer systems} displays all transfer systems $\cT$ for the symmetric group $\Sigma_3$. The table also displays $\mathrm{Hull}(\cT)$ and whether the pair $(\mathrm{Hull}(\cT),\cT)$ is realizable using our methods.  Rows $1$, $5$, $8$, and $9$ are realizable as Steiner and linear isometries pairs, see \cite[Examples 4.9 and 5.11]{RubinDetecting}.  Row $2$ is realizable by applying \cref{corollary: realizability is closed under meets} to rows $5$ and $8$.  Finally, row $3$ is realizable since it is the $H_1$-local transfer system for $\Sigma_3$.  Using this, one can check that at least 24 out of 36 compatible pairs are realizable.  We note that if one could show that row $6$ was realizable, then we would immediately obtain that row $7$ is realizable, by applying \cref{corollary: realizability is closed under meets} to rows $6$ and $8$.  Thus the two outstanding cases are rows $4$ and $6$. 
    \end{example}

\begin{table}[ht!]
    \begin{tabular}{ccccc}
        \toprule
         & & $\cT$ & \hspace*{1cm} $\Hull(\cT)$ \hspace*{1cm} & Realizable? 
            \\ \midrule
        1 & &\sthreetriv & \sthreetriv & $\checkmark$
            \\[1cm]
        2 & &\sthreea & \sthreea & $\checkmark$
            \\[1cm]
        3 & &\sthreeb & \sthreeb & $\checkmark$
            \\[1cm]
        4 & &\sthreeab & \sthreeab & ?
            \\[1cm]
        5 & &\sthreec & \sthreec & $\checkmark$
            \\[1cm]
        6 & &\sthreeabc & \sthreeab & ?
            \\[1cm]
        7 & &\sthreed & \sthreeab & ?
            \\[1cm]
        8 & &\sthreee & \sthreeab & $\checkmark$
            \\[1cm]
        9 & &\sthreeall & \sthreeall & $\checkmark$
            \\ \bottomrule
    \end{tabular}
    \medskip
    \caption{An enumeration of all $\Sigma_3$-transfer systems $\cT$, their hulls $\Hull(\cT)$, and whether or not the pair $(\Hull(\cT), \cT)$ is realizable using our methods. Here, $H_1$, $H_2$, and $H_3$ are the three conjugate copies of $C_2$ inside $\Sigma_3$.}
    \label{table: S3 transfer systems}
\end{table}

\clearpage

\bibliographystyle{alpha}
\bibliography{references}
\end{document}